\documentclass[letterpaper, 10 pt, conference]{ieeeconf}  

\IEEEoverridecommandlockouts                              
\def\bsx{{\boldsymbol{x}}}
\def\bsy{{\boldsymbol{y}}}

\def\calA{{\mathcal{A}}}

\def\calP{{\mathcal{P}}}

\def\calU{{\mathcal{U}}}

\def\calX{{\mathcal{X}}}

\def\Real{{\mathbb{R}}}

\let\labelindent\relax
\usepackage{epsfig} 
\usepackage{epstopdf}

\usepackage{fancyhdr}
\usepackage{indentfirst}
\usepackage{amsmath}
\usepackage{amssymb}
\usepackage{algorithm}
\usepackage{algorithmic}
\usepackage{enumitem}

\usepackage{amsthm}

\usepackage{graphicx}
\usepackage{color}

\definecolor{gray}{RGB}{128,128,128}

\usepackage[short]{optidef}
\usepackage{latexsym}
\usepackage{amsfonts}
\usepackage{amssymb,amsmath,amsfonts}
\usepackage{amsmath,mathrsfs,bm,url,times}
\usepackage{cite}
\usepackage{caption}
\usepackage{subcaption}
\newtheorem{assumption}{Assumption}

\newtheorem{theorem}{Theorem}
\newtheorem{lemma}{Lemma}

\newtheorem{remark}{Remark}
\newtheorem{corollary}{Corollary}

\DeclareMathOperator*{\argmin}{arg\,min}

\DeclareMathOperator{\Reg}{Reg}
\DeclareMathOperator{\col}{col}
\DeclareMathOperator{\inout}{in}
\DeclareMathOperator{\outin}{out}
\usepackage{array}
\newcolumntype{M}[1]{>{\centering\arraybackslash}m{#1}}
\newcolumntype{N}{@{}m{0pt}@{}}

\allowdisplaybreaks
\title{\LARGE \bf
Distributed Bandit Online Convex Optimization\\ with Time-Varying Coupled Inequality Constraints
}

\author{Xinlei Yi, Xiuxian Li, Tao Yang, Lihua Xie, Karl H. Johansson, and Tianyou Chai
\thanks{This work was supported by the
Knut and Alice Wallenberg Foundation, the  Swedish Foundation for Strategic Research, the Swedish Research Council, and Ministry of Education of Republic of Singapore under Grant MoE Tier 1 RG72/19.}
\thanks{X. Yi and K. H. Johansson are with the Division of Decision and Control Systems, School of Electrical Engineering and Computer Science, KTH Royal Institute of Technology, 100 44, Stockholm, Sweden.
        {\tt\small \{xinleiy, kallej\}@kth.se}.}%
\thanks{X. Li and L. Xie are with School of Electrical and Electronic Engineering,
Nanyang Technological University, 50 Nanyang Avenue, Singapore 639798. {\tt\small \{xiuxianli, elhxie\}@ntu.edu.sg}.}
\thanks{T. Yang and T. Chai are with the State Key Laboratory of Synthetical Automation for Process Industries, Northeastern University, 110819, Shenyang, China. {\tt\small \{yangtao,tychai\}@mail.neu.edu.cn}.}
}

\begin{document}

\maketitle
\thispagestyle{empty}
\pagestyle{empty}

\begin{abstract}\label{cdc18so:Abstract}
This paper considers the problem of distributed bandit online convex optimization with time-varying coupled inequality constraints. This problem can be defined as a repeated game between a group of learners and an adversary. The learners attempt to minimize a sequence of global loss functions and at the same time satisfy a sequence of coupled constraint functions. The global loss and the coupled constraint functions are the sum of local convex loss and constraint functions, respectively, which are adaptively generated by the adversary. The local loss and constraint functions are revealed in a bandit manner, i.e., only the values of loss and constraint functions at sampled points are revealed to the learners, and the revealed function values are held privately by each learner. We consider two scenarios, one- and two-point bandit feedback, and propose two corresponding distributed bandit online algorithms used by the learners. We show that sublinear expected regret and constraint violation are achieved by these two algorithms, if the accumulated variation of the comparator sequence also grows sublinearly. In particular, we show that $\mathcal{O}(T^{\theta_1})$ expected static regret and $\mathcal{O}(T^{7/4-\theta_1})$ constraint violation are achieved in the one-point bandit feedback setting, and  $\mathcal{O}(T^{\max\{\kappa,1-\kappa\}})$ expected static regret and $\mathcal{O}(T^{1-\kappa/2})$ constraint violation in the two-point bandit feedback setting, where $\theta_1\in(3/4,5/6]$ and $\kappa\in(0,1)$ are user-defined trade-off parameters. Finally, these theoretical results are illustrated by numerical simulations of a simple power grid example.

\emph{Index Terms}---Bandit convex optimization, distributed optimization, gradient approximation, online optimization, time-varying constraints
\end{abstract}
\section{Introduction}\label{dbco:introduction}
Online convex optimization is a promising methodology for modeling sequential tasks and has important applications in machine learning \cite{shalev2012online}, smart grids \cite{zhou2017incentive}, sensor networks \cite{shahrampour2017distributed,yuan2017adaptive}, and so on. It can be traced back to the 1990s \cite{cesa1996worst,gentile1999linear,gordon1999regret,zinkevich2003online}. Online convex optimization can be understood as a repeated game between a learner and an adversary \cite{shalev2012online}. At round $t$ of the game, the learner chooses a point $x_t$ from a known convex set $\mathbb{X}$. Then, the adversary observes $x_t$ and chooses a convex loss function $f_t:\mathbb{X}\rightarrow \mathbb{R}$. After that, the loss function $f_t$ is revealed to the learner who suffers a loss $f_t(x_t)$. Note that at each round the loss function can be arbitrarily chosen by the adversary, especially with no probabilistic model imposed on the choices, which is the key difference between online and stochastic convex optimization. Such an adversary with the power to arbitrarily choose the loss functions is said to be a completely adaptive adversary \cite{agarwal2010optimal}. The goal of the learner is to choose a sequence $(x_1,\dots,x_T)$ such that her regret\footnote{In the literature, this is called static regret. Another commonly used metric is the dynamic regret $\sum_{t=1}^{T}f_t(x_t)-\sum_{t=1}^{T}\min_{y_t\in\mathbb{X}}f_t(y_t)$. Actually, analysis for dynamic regret is an intermediate step towards the analysis for static regret.} $\sum_{t=1}^{T}f_t(x_t)-\min_{y\in\mathbb{X}}\sum_{t=1}^{T}f_t(y)$ is minimized, where $T$ is the total number of rounds.  Over the past two decades, online convex optimization has been extensively studied, e.g., \cite{shalev2012online,shahrampour2017distributed,yuan2017adaptive,zinkevich2003online,hazan2007logarithmic,jenatton2016adaptive,sun2017safety,neely2017online,yu2017online,yu2016low,wei2019online,sadeghi2019online}. All existing online algorithms require the knowledge of the entire loss function or the gradient of the loss function. In particular, it is known that projection-based online gradient descent algorithm achieves an $\mathcal{O}(\sqrt{T})$ static regret bound for loss functions with bounded subgradients and that this is a tight bound up to constant factors \cite{hazan2007logarithmic}.


Bandit online convex optimization is online convex optimization with bandit feedback, i.e., at each round only the values of the loss functions are revealed, rather than the entire loss function, the gradient of the loss function, or some other information. Bandit feedback is suitable to model various applications, where the entire function or gradient information is not available, such as online source localization, online routing in data networks, and online advertisement placement in web search \cite{hazan2016introduction}. For such applications, existing online algorithms are inapplicable but gradient-free (zeroth-order) optimization methods are needed.
Gradient-free optimization methods  have a long history \cite{matyas1965random} and have an evident advantage since computing a function value is much simpler than computing its gradient. Gradient-free optimization methods have gained renewed interests in recent years, e.g., \cite{nesterov2017random,yuan2014randomized,pang2019randomized,tang2019distributed}. Essentially, bandit online convex optimization is a gradient-free method to solve convex optimization problems.  However, in a bandit setting, the sublinear static regret bound may not be guaranteed if the adversary still can arbitrarily choose the loss function. Under completely adaptive adversary, the authors of \cite{agarwal2010optimal} gave an example to show that any algorithm suffer at least a linear regret. Therefore, the power of the adversary should be limited. The adversary chooses $f_t$ based only on the learner's past decisions $x_1,\dots,x_{t-1}$, but not on her current decision $x_t$. In other words, the adversary chooses $f_t$ at the beginning of round $t$, before the learner choose her decision. The adversary with such a limited power is said to be an adaptive adversary \cite{agarwal2010optimal}. 

A key step in bandit online convex optimization is to estimate the gradient of the loss function by sampling the loss function.
Various algorithms have been developed  and can be divided into two categories depending on the number of samplings. Algorithms with one sampling at each round have been proposed in \cite{flaxman2005online,dani2008price,abernethy2009competing,abernethy2012interior,saha2011improved,hazan2014bandit,bubeck2015bandit,bubeck2016multi,hazan2016optimal}. Specifically, in \cite{flaxman2005online}, $\mathcal{O}(T^{3/4})$ expected static regret was achieved for  Lipschitz-continuous functions. Better regret bounds can be guaranteed if additional assumptions are made. The authors of \cite{dani2008price} considered linear loss functions and achieved $\mathcal{O}(\sqrt{T})$ expected static regret. The authors of  \cite{abernethy2009competing,abernethy2012interior} also considered linear loss functions and proposed algorithms that achieved $\mathcal{O}(\sqrt{T\log(T)})$ expected static regret. The authors of \cite{saha2011improved} studied smooth loss functions and achieved $\mathcal{O}(T^{2/3}(\log(T))^{1/3})$ expected static regret. The authors of \cite{hazan2014bandit} considered strongly convex and smooth loss functions and achieved $\mathcal{O}(\sqrt{T\log(T)})$ expected static regret. One common assumption in \cite{abernethy2009competing,abernethy2012interior,saha2011improved,hazan2014bandit} is that the convex domain admits a self-concordant barrier. The authors of \cite{bubeck2015bandit} showed that $\mathcal{O}(\sqrt{T}\log(T))$ expected static regret can be achieved for Lipschitz-continuous loss functions with one-dimensional domains, but they did not develop any explicit algorithm. This result was extended to arbitrary dimensions in \cite{bubeck2016multi}, but still without any explicit algorithm. Based on the ellipsoid method to online learning, the authors of \cite{hazan2016optimal} proposed an algorithm for Lipschitz-continuous loss functions and achieved $\mathcal{O}(\sqrt{T}\log(T))$ expected static regret.

Algorithms with two or more samplings at each round have been proposed in \cite{agarwal2010optimal,duchi2015optimal,shamir2017optimal,yi2016tracking,tatarenko2018minimizing,shames2019online}. The expected static regret bounds can then be reduced compared to the one-sample case. The authors of \cite{agarwal2010optimal} extended the one-point sampling bandit algorithm proposed in \cite{flaxman2005online} to a two-point sampling algorithm and obtained $\mathcal{O}(\log(T))$ expected static regret for Lipschitz-continuous and strongly convex loss functions. Moreover, with $p+1$ samplings at each round, where $p$ is the state dimension, they proposed a deterministic algorithm and showed that the algorithm can achieve $\mathcal{O}(\sqrt{T})$ regret for Lipschitz-continuous and smooth loss functions, and $\mathcal{O}(\log(T))$ expected static regret for strongly convex and smooth loss functions. The author of \cite{shamir2017optimal} proposed a simple algorithm with two samplings at each round and obtained $\mathcal{O}(\sqrt{T})$ expected static regret for Lipschitz-continuous loss functions. Without assuming that the decision set is bounded, the author of \cite{tatarenko2018minimizing} proposed a class of algorithms with one or two samplings at each round and obtained $\mathcal{O}(T^{2/3})$ and $\mathcal{O}(\sqrt{T})$ expected static regrets, respectively, for smooth loss functions.

Aforementioned studies did not consider equality or inequality constraints. In the literature, there are few papers considering bandit online convex  optimization with such constraints, although such constraints are common in applications. The authors of \cite{mahdavi2012trading} studied online convex optimization with static inequality constraints and bandit feedback for constraints, while the authors of \cite{chen2018bandit} studied online convex optimization with time-varying inequality constraints and bandit feedback for loss functions. The authors of \cite{cao2019online} studied online convex optimization with time-varying inequality constraints and bandit feedback for both loss and constraint functions.

Most existing bandit online convex optimization studies are in a centralized setting and
only few papers considered distributed bandit online convex optimization. When loss functions are strongly convex, the authors of \cite{yuan2016online} proposed a consensus-based distributed bandit online algorithm with one sampling at each round and obtained $\mathcal{O}(\sqrt{T}\log(T))$ expected static regret. When loss functions are quadratic, the authors of \cite{yuan2019distributed} proposed a consensus-based distributed bandit online algorithm with two samplings at each round and obtained  $\mathcal{O}(\sqrt{T})$ expected static regret when there are set constraints. When there are static linear inequality constraints, they also established $\mathcal{O}(T^{\max\{\beta,1-\beta\}})$ and $\mathcal{O}(T^{1-\beta/2})$ bounds on the expected static regret and constraint violation, respectively, where $\beta\in(0,1)$ is a user-defined trade-off parameter.

This paper considers the problem of distributed bandit online convex optimization with time-varying coupled inequality constraints. This problem can be interpreted as a repeated game between a group of learners and an adversary. The learners attempt to minimize a sequence of global loss functions and at the same time satisfy a sequence of coupled constraint functions. The global loss and the coupled constraint functions are the sum of local convex loss and constraint functions, respectively. They are generated adaptively by the adversary. The local loss and constraint functions are revealed in a bandit manner and the revealed information is held privately by each learner. Specifically, at each round each learner can sample her local loss and constraint function at one point (i.e., one-point bandit feedback) or two points (i.e., two-point bandit feedback). Compared to existing studies, the contributions of this paper are summarized as follows.

In the one-point bandit feedback setting, we propose a distributed bandit online algorithm with a one-point sampling gradient estimator to solve the considered problem. To the best of our knowledge, this is the first algorithm to solve the online convex optimization problem with time-varying inequality constraints in the one-point bandit feedback setting. An advantage of our algorithm is that the total number of rounds is not used in the algorithm, which is an improvement compared to the one-point sampling algorithms in \cite{flaxman2005online,abernethy2009competing,abernethy2012interior,saha2011improved,hazan2014bandit,chen2018bandit,yuan2016online}, although these paper did not consider bandit feedback for the time-varying inequality constraints or did not even consider time-varying inequality constraints at all. Sublinear expected regret and constraint violation are achieved by this algorithm if the accumulated variation of the comparator sequence also grows sublinearly. In particular, $\mathcal{O}(T^{\theta_1})$ expected static regret and $\mathcal{O}(T^{7/4-\theta_1})$ constraint violation are achieved, where $\theta_1\in(3/4,5/6]$ is a user-defined trade-off parameter. Specifically, when there are no inequality constraints, the proposed algorithm achieves $\mathcal{O}(T^{3/4})$ expected static regret. The same expected static regret bound has been achieved by the one-point sampling algorithm in \cite{flaxman2005online}. However, in \cite{flaxman2005online} the total number of iterations as well as the Lipschitz constant and upper bound of the loss functions are needed.

In the two-point bandit feedback setting, we propose a distributed bandit online algorithm with a two-point sampling gradient estimator. This algorithm does not require the total number of rounds  or any other parameters related to the loss or constraint functions, which is different from the two-point sampling algorithms in \cite{agarwal2010optimal,duchi2015optimal,shamir2017optimal,yi2016tracking,shames2019online,mahdavi2012trading,chen2018bandit,cao2019online,yuan2019distributed}.
In an average sense, this algorithm is as efficient as the algorithms proposed in \cite{jenatton2016adaptive,sun2017safety,mahdavi2012trading,yi2019distributed}, although \cite{jenatton2016adaptive,sun2017safety,yi2019distributed} are in a full-information feedback setting and \cite{mahdavi2012trading} considers bandit setting only for the constraint functions. In particular, $\mathcal{O}(T^{\max\{\kappa,1-\kappa\}})$ expected static regret and $\mathcal{O}(T^{1-\kappa/2})$ constraint violation are achieved by our algorithm, where $\kappa\in(0,1)$ is a user-defined parameter. Compared with the bandit algorithm in \cite{shamir2017optimal}, which achieved $\mathcal{O}(\sqrt{T})$ expected static regret under static set constraints and centralized computations using the total number of rounds as well as the Lipschitz constant of the loss function, we relax all these assumptions.

The rest of this paper is organized as follows. Section~\ref{dbco:preliminary} introduces the preliminaries. Section~\ref{dbco:problem} gives the problem formulation and a motivating example. Sections~\ref{dbco:mainresult1} and \ref{dbco:mainresult2} provide the distributed bandit online algorithms for one- and two-point bandit feedback, respectively, and present their expected regret and constraint violation bounds. Section~\ref{dbco:simulation} gives numerical simulations for the motivating example. Finally, Section~\ref{dbco:conclusions} concludes the paper. Proofs are given in the Appendix.

\noindent {\bf Notations}: All inequalities and equalities are understood componentwise. $\mathbb{R}^p$ and $\mathbb{R}^p_+$ denote the set of $p$-dimensional vectors and nonnegative vectors, respectively. $\mathbb{N}_+$ stands for the set of positive integers. $[n]$ represents the set $\{1,\dots,n\}$ for any $n\in\mathbb{N}_+$.
$[x]_j$ is the $j$-th element of a vector $x\in\mathbb{R}^p$.
$\langle x,y\rangle$ denotes the standard inner product of two vectors $x$ and $y$. $x^\top$ stands for the transpose of the vector or matrix $x$.
$\|\cdot\|$ ($\|\cdot\|_1$) represents the Euclidean norm (1-norm) for vectors and the induced 2-norm (1-norm) for matrices.
$\mathbb{B}^p$ and $\mathbb{S}^p$ are the unit ball and sphere centered around the origin in $\mathbb{R}^p$ under Euclidean norm, respectively.
${\bf I}_n$ denotes the $n$-dimensional identity matrix.
${\bf 1}_n$ (${\bf 0}_n$) stands for the column one (zero)
vector of dimension $n$. $\col(z_1,\dots,z_k)$ represents the concatenated column vector of vectors $z_i\in\mathbb{R}^{n_i},~i\in[k]$.
$\log(\cdot)$ is the natural logarithm. Given two scalar sequences $\{\alpha_t,~t\in\mathbb{N}_+\}$ and $\{\beta_t>0,~t\in\mathbb{N}_+\}$, $\alpha_t=\mathcal{O}(\beta_t)$ means that $\limsup_{t\rightarrow\infty}(\alpha_t/\beta_t)$ is bounded, while $\alpha_t=\mathbf{o}(\beta_t)$  means that $\lim_{t\rightarrow\infty}(\alpha_t/\beta_t)=0$. For a set $\mathbb{K}\subseteq\mathbb{R}^p$, $\calP_{\mathbb{K}}(\cdot)$ denotes the projection operator, i.e.,  $\calP_{\mathbb{K}}(x)=\argmin_{y\in\mathbb{K}}\|x-y\|^2,~\forall x\in\Real^{p}$. For simplicity, $[\cdot]_+$ is used to denote $\calP_{\mathbb{K}}(\cdot)$ when $\mathbb{K}=\mathbb{R}^p_+$.

\section{Preliminaries}\label{dbco:preliminary}
In this section, we present some definitions and properties related to graph theory and gradient approximation.

\subsection{Graph Theory}
Let $\mathcal{G}_t=(\mathcal{V},\mathcal{E}_t)$ denote a time-varying directed graph, where $\mathcal{V}=[n]$ is the agent set and $\mathcal{E}_t\subseteq\mathcal{V}\times\mathcal{V}$ is the edge set. A directed edge $(j,i)\in\mathcal{E}_t$ means that agent $i$ can receive data from agent $j$ at time $t$. Let $\mathcal{N}^{\inout}_i(\mathcal{G}_t)=\{j\in [n]\mid (j,i)\in\mathcal{E}_t\}$ and $\mathcal{N}^{\outin}_i(\mathcal{G}_t)=\{j\in [n]\mid (i,j)\in\mathcal{E}_t\}$ be the sets of in- and out-neighbors, respectively, of agent $i$ at time $t$. A directed path is a sequence of consecutive directed edges. A  directed graph is said to be strongly connected if there is at least one directed path
from any agent to any other agent in the graph. The mixing matrix $W_t\in\mathbb{R}^{n\times n}$ at time $t$ fulfills $[W_t]_{ij}>0$ if $(j,i)\in\mathcal{E}_t$ or $i=j$, and $[W_t]_{ij}=0$ otherwise.

\subsection{Gradient Approximation}
In this section, we introduce one- and two-point sampling gradient estimators.

Let $f:\mathbb{K}\rightarrow\mathbb{R}$ be a function with $\mathbb{K}\subset\mathbb{R}^p$. We assume that $\mathbb{K}$ is convex and bounded, and has a nonempty interior. Specifically, we assume that $\mathbb{K}$ contains the ball of radius $r(\mathbb{K})$ centered at the origin and is contained in the ball of radius $R(\mathbb{K})$, i.e., $r(\mathbb{K})\mathbb{B}^p\subseteq\mathbb{K}\subseteq R(\mathbb{K})\mathbb{B}^p$. The authors of \cite{flaxman2005online} proposed the following gradient estimator,
\begin{align}\label{dbco:gradient:model1}
\hat{\nabla}_1f(x)=\frac{p}{\delta}f(x+\delta u)u,~\forall x\in(1-\xi)\mathbb{K},
\end{align}
where $u\in\mathbb{S}^p$ is a uniformly distributed random vector, $\delta\in(0,r(\mathbb{K})\xi]$ is an exploration parameter, and $\xi\in(0,1)$ is a shrinkage coefficient. The estimator $\hat{\nabla}_1f$ only requires to sample the function at one point, so it is a one-point sampling gradient estimator. Some intuition for this estimator can be found in \cite{flaxman2005online}. Different from \cite{nesterov2017random}, uniform distribution rather than Gaussian distribution is used to generate $u$ in \eqref{dbco:gradient:model1} since the later may generate unbounded $u$.  The estimator  $\hat{\nabla}_1f$ is defined over the set $(1-\xi)\mathbb{K}$ instead of $\mathbb{K}$, since otherwise the perturbations may move points outside $\mathbb{K}$. The feasibility of the perturbations is guaranteed by the following lemma.
\begin{lemma}\label{dbco:lemma:flaxman}
(Observation 2 in \cite{flaxman2005online}) For any $x\in(1-\xi)\mathbb{K}$ and $u\in\mathbb{S}^p$, it holds that $x+\delta u\in\mathbb{K}$ for any $\delta\in(0,r(\mathbb{K})\xi]$.
\end{lemma}

For our two-point sampling gradient estimator, we use
\begin{align}\label{dbco:gradient:model2}
\hat{\nabla}_2f(x)=\frac{p}{\delta}(f(x+\delta u)-f(x))u,~\forall x\in(1-\xi)\mathbb{K}.
\end{align}
The intuition follows from directional derivatives \cite{duchi2015optimal}. 

Both estimators $\hat{\nabla}_1f$ and $\hat{\nabla}_2f$ are unbiased gradient estimators of $\hat{f}$, where $\hat{f}$ is the uniformly smoothed version of $f$ defined as
\begin{align*}
\hat{f}(x)=\mathbf{E}_{v\in\mathbb{B}^p}[f(x+\delta v)],~\forall x\in(1-\xi)\mathbb{K},
\end{align*}
where the expectation is with respect to uniform distribution.
Some properties of $\hat{f}$, $\hat{\nabla}_1f$, and $\hat{\nabla}_2f$ are presented in the following lemma.
\begin{lemma}\label{dbco:lemma:uniformsmoothing}
\begin{enumerate}[label=(\alph*)]
\item The uniform smoothing $\hat{f}$ is differentiable on $(1-\xi)\mathbb{K}$ even when $f$ is not and for all $x\in(1-\xi)\mathbb{K}$,
\begin{align*}
\nabla \hat{f}(x)=\mathbf{E}_{u\in\mathbb{S}^p}[\hat{\nabla}_1f(x)]=\mathbf{E}_{u\in\mathbb{S}^p}[\hat{\nabla}_2f(x)].
\end{align*}
\item If $f$ is convex on $\mathbb{K}$, then $\hat{f}$ is convex on $(1-\xi)\mathbb{K}$ and
\begin{align*}
f(x)\le \hat{f}(x),~\forall x\in(1-\xi)\mathbb{K}.
\end{align*}
\item If $f$ is Lipschitz-continuous on $\mathbb{K}$ with constant $L_0(f)>0$, then $\hat{f}$ and $\nabla \hat{f}$ are Lipschitz-continuous on $(1-\xi)\mathbb{K}$ with constants $L_0(f)$ and $pL_0(f)/\delta$, respectively. Moreover,
\begin{align*}
|\hat{f}(x)-f(x)|\le\delta L_0(f),~\forall x\in(1-\xi)\mathbb{K}.
\end{align*}
\item If $f$ is bounded on $\mathbb{K}$, i.e., there exists $F_0(f)>0$ such that $|f(x)|\le F_0(f),~\forall x\in\mathbb{K}$, then
\begin{align*}
&|\hat{f}(x)|\le F_0(f),\\
&\|\hat{\nabla}_1f(x)\|\le\frac{pF_0(f)}{\delta},~\forall x\in(1-\xi)\mathbb{K}.
\end{align*}
\item If $f$ is Lipschitz-continuous on $\mathbb{K}$ with constant $L_0(f)>0$, then
\begin{align*}
\|\hat{\nabla}_2f(x)\|\le pL_0(f),~\forall x\in(1-\xi)\mathbb{K}.
\end{align*}
\end{enumerate}
\end{lemma}
\begin{proof}
See Appendix~\ref{dbco:lemma:uniformsmoothinproof}.
\end{proof}

Intuitively, the key idea of gradient-free methods is using the smoothed function $\hat{f}$ to replace the original function $f$ since they are close when $\delta$ is small as shown in (c) of Lemma~\ref{dbco:lemma:uniformsmoothing}. Moreover, the gradient of $\hat{f}$ can be estimated by the gradient estimators $\hat{\nabla}_1f$ or $\hat{\nabla}_2f$ as shown in (a). The main difference between these two gradient estimators is that the norm of $\hat{\nabla}_1f$ is large when $\delta$ is small, while $\hat{\nabla}_2f$ has a bounded norm, as shown in (d) and (e), respectively. This difference leads to improved results for the two-point bandit feedback algorithm, as will be seen in the later sections.

\section{Problem Formulation}\label{dbco:problem}
We consider the problem of distributed bandit online convex optimization with time-varying coupled inequality constraints. This problem can be defined as a repeated game between a group of $n$ learners indexed by $i\in[n]$ and an adversary. At round $t$ of the game, the adversary first arbitrarily chooses $n$ local loss functions $\{f_{i,t}:\mathbb{X}_i\rightarrow \mathbb{R},~i\in[n]\}$ and $n$ local constraint functions $\{g_{i,t}:\mathbb{X}_i\rightarrow \mathbb{R}^m,~i\in[n]\}$, where each $\mathbb{X}_i\subseteq\mathbb{R}^{p_i}$ is a known closed convex set with $p_i$ and $m$ being positive integers. Then, without knowing $\{f_{i,t},~i\in[n]\}$ and $\{g_{i,t},~i\in[n]\}$, all learners simultaneously choose their decisions $\{x_{i,t}\in\mathbb{X}_i,~i\in[n]\}$. Each learner~$i$ samples the values of $f_{i,t}$ and $g_{i,t}$ at the point $x_{i,t}$ as well as other potential points, i.e., the learners receive bandit feedback from the adversary. These values are held privately by each learner. At the same moment, the learners exchange data with their neighbors over a time-varying directed graph $\mathcal{G}_t$. The goal of the learners is to cooperatively choose a global decision sequence $\bsx_T=(x_{1},\dots,x_{T})$, where $T$ is the total number of rounds and $x_t=\col(x_{1,t},\dots,x_{n,t})$, such that the accumulated global loss $\sum_{t=1}^{T}f_t(x_t)$ is competitive with the loss of any comparator sequence $\bsy_T=(y_{1},\dots,y_{T})$ with $y_t=\col(y_{1,t},\dots,y_{n,t})$ (i.e., the regret is as small as possible) and at the same time the constraint violation is as small as possible, where $f_t(x_t)=\sum_{i=1}^nf_{i,t}(x_{i,t})$ is the global loss function.

Specifically, the regret of a global decision sequence $\bsx_T$ with respect to a comparator sequence $\bsy_T$ is defined as
\begin{align*}
\Reg(\bsx_T,\bsy_T)=\sum_{t=1}^Tf_{t}(x_{t})-\sum_{t=1}^Tf_{t}(y_{t}).
\end{align*}
In the literature, there are two commonly used comparator sequences. One is  the optimal dynamic decision sequence in hindsight $\bsy_T=\bsx^*_T=(x^*_{1},\dots,x^*_{T})$ solving the constrained convex optimization problem
\begin{mini}
{x_t\in \mathbb{X}}{\sum_{t=1}^{T}f_t(x_t)}{\label{online:problem1}}{}
\addConstraint{g_t(x_t)\le}{{\bf 0}_m,\quad}{\forall t\in[T],}
\end{mini}
where $\mathbb{X}=\mathbb{X}_1\times\cdots\times \mathbb{X}_n\subseteq\mathbb{R}^{p}$ is the global decision set, $p=\sum_{i=1}^np_i$, and $g_t(x_t)=\sum_{i=1}^ng_{i,t}(x_{i,t})$ is the coupled constraint function.
In order to guarantee that problem (\ref{online:problem1}) is feasible, we assume that for any $T\in\mathbb{N}_+$, the set of all feasible decision sequences $
\mathcal{X}_{T}=\{(x_1,\dots,x_T):~x_t\in \mathbb{X},~g_{t}(x_t)\le{\bf0}_{m},~
t\in[T]\}
$ is non-empty. With this standing assumption, an optimal dynamic decision sequence to (\ref{online:problem1}) always exists. In this case $\Reg(\bsx_T,\bsx^*_T)$ is called the dynamic regret for $\bsx_T$.
Another comparator sequence is $\bsy_T=\check{\bsx}^*_T=(\check{x}^*_T,\dots,\check{x}^*_T)$, where $\check{x}^*_T$ is the optimal static decision in hindsight solving
\begin{mini}
{x\in \mathbb{X}}{\sum_{t=1}^{T}f_t(x)}{\label{online:problem1static}}{}
\addConstraint{g_t(x)\le}{{\bf 0}_m,\quad}{\forall t\in[T].}
\end{mini}
Similar to above, in order to guarantee that problem (\ref{online:problem1static}) is feasible, we assume that for any $T\in\mathbb{N}_+$, the set of all feasible static decision sequences
$\check{\calX}_{T}=\{(x,\dots,x):~x\in \mathbb{X},~g_{t}(x)\le{\bf0}_{m},~t=1,\dots,T\}\subseteq\calX_{T}$ is non-empty.
In this case $\Reg(\bsx_T,\check{\bsx}^*_T)$ is called the static regret. It is straightforward to see that $\Reg(\bsx_T,\bsy_T)\le\Reg(\bsx_T,\bsx^*_T),~\forall \bsy_T\in\calX_{T}$, and that $\Reg(\bsx_T,\check{\bsx}^*_T)\le\Reg(\bsx_T,\bsx^*_T)$.

For a decision sequence $\bsx_T$, the constraint violation is defined as
\begin{align*}
\|[\sum_{t=1}^Tg_{t}(x_{t})]_+\|.
\end{align*}
Note that this definition implicitly allows constraint violations at some times to be compensated by strictly feasible decisions at other times. This is appropriate for constraints that have a cumulative nature such as energy budgets enforced through average power constraints.

The considered problem can be viewed as an extension of the problem studied in \cite{yi2019distributed}, from full information feedback to bandit feedback. As discussed in Section~\ref{dbco:introduction}, two main motivations of considering bandit feedback are that (1) gradient information is not available
in many applications \cite{hazan2016introduction}; and (2) computing a function value is much simpler than computing its gradient \cite{nesterov2017random}. 

We make the following assumptions on the time-varying directed graph $\mathcal{G}_t$ as well as the loss and constraint functions.

\begin{assumption}\label{dbco:assgraph}
For any $t\in\mathbb{N}_+$, the directed graph $\mathcal{G}_t$ satisfies the following conditions:
\begin{enumerate}[label=(\alph*)]
  \item There exists a constant $w\in(0,1)$, such that $[W_t]_{ij}\ge w$ if $[W_t]_{ij}>0$.
  \item The mixing matrix $W_t$ is doubly stochastic, i.e., $\sum_{i=1}^n[W_t]_{ij}=\sum_{j=1}^n[W_t]_{ij}=1,~\forall i,j\in[n]$.
  \item There exists an integer $\iota>0$ such that the directed graph $(\mathcal{V},\cup_{l=0,\dots,\iota-1}\mathcal{E}_{t+l})$ is strongly connected.
\end{enumerate}
\end{assumption}

\begin{assumption}\label{dbco:assfunction:function}
\begin{enumerate}[label=(\alph*)]
  \item For each $i\in[n]$, the set $\mathbb{X}_i$ is convex and closed. Moreover, there exist $r_i>0$ and $R_i>0$ such that
      \begin{align}
      r_i\mathbb{B}^{p_i}\subseteq \mathbb{X}_i\subseteq R_i\mathbb{B}^{p_i}.\label{dbco:domainupper}
      \end{align}
  \item For each $i\in[n]$, $\{f_{i,t}(x)\}$  and $\{[g_{i,t}(x)]_j,~j\in[m]\}$ are convex and uniformly bounded on $\mathbb{X}_i$, i.e., there exist constants $F_{f_i}>0$ and $F_{g_i}>0$ such that for all $t\in\mathbb{N}_+,~j\in[m],~x\in \mathbb{X}_i$,
\begin{align}
|f_{i,t}(x)|\le F_{f_i},~\text{and}~|[g_{i,t}(x)]_j|\le F_{g_i}.\label{dbco:assfunction:ftgtupper}
\end{align}
  \item For each $i\in[n]$, $f_{i,t}$ and $g_{i,t}$ are differentiable. Moreover, $\{\nabla f_{i,t}\}$  and $\{\nabla [g_{i,t}(x)]_j,~j\in[m]\}$ are uniformly bounded on $\mathbb{X}_i$, i.e., there exist constants $G_{f_i}>0$ and $G_{g_i}>0$ such that for all $t\in\mathbb{N}_+,~j\in[m],~x\in \mathbb{X}_i$,
  \begin{align}
\|\nabla f_{i,t}(x)\|\le G_{f_i},~\text{and}~\|\nabla [g_{i,t}(x)]_j\|\le G_{g_i}.\label{dbco:assfunction:subgupper}
\end{align}
\end{enumerate}
\end{assumption}
Assumption~\ref{dbco:assgraph} is a mild assumption and common in the literature of distributed optimization. Assumption~\ref{dbco:assfunction:function} appears often in the literature of bandit online convex optimization.
From Assumption~\ref{dbco:assfunction:function} and Lemma~2.6 in \cite{shalev2012online}, it follows that for all $t\in\mathbb{N}_+,~i\in[n],~j\in[m],~x,~y\in \mathbb{X}_i$,
\begin{subequations}
\begin{align}
&|f_{i,t}(x)-f_{i,t}(y)|\le G_{f_i}\|x-y\|,\label{dbco:assfunction:functionLipf}\\
      &|[g_{i,t}(x)]_j-[g_{i,t}(y)]_j|\le G_{g_i}\|x-y\|,\label{dbco:assfunction:functionLipg}
\end{align}
\end{subequations}
i.e., $\{f_{i,t}(x)\}$  and $\{[g_{i,t}(x)]_j\}$ are Lipschitz-continuous on $\mathbb{X}_i$ with constants $G_{f_i}$ and $G_{g_i}$, respectively.

\subsection{Motivating Example}\label{dbco:example}
As a motivating example, consider a power grid with $n$ power generation units. Each unit $i$ has $p_i$ conventional and renewable power generators. The units can communicate through the information infrastructure. At stage $t$, let $x_{i,t}\in\mathbb{X}_i$ and $\mathbb{X}_i\subset\mathbb{R}^{p_i}$ be the output and the set of feasible outputs of the generators in unit $i$, respectively. To generate the output, each unit $i$ suffers a cost $f_{i,t}(x_{i,t})$. This local cost $f_{i,t}$ is usually described by a quadratic function \cite{abdelaziz2016combined}, but its accurate form is unknown in advance, since fossil fuel price is fluctuating and renewable energy is highly uncertain and unpredictable. Except the local generator limit constraints $\mathbb{X}_i$, all units need to cooperatively take into account global constraints, such as power balance and emission constraints. The global constraints can be modelled as $\sum_{i=1}^ng_{i,t}(x_{i,t})\le{\bf 0}_m$, where  $g_{i,t}$ is unit $i$'s local constraint function. Again, the precise form of the constraint functions is unknown in advance either since that power demands can change from one hour to the next, or that the emission can change  due to the uncertain and unpredictable features of renewable energy. The goal of the units is to reduce the global cost while satisfying the constraints.

\section{One-Point Bandit Feedback}\label{dbco:mainresult1}
In this section, we propose a  distributed bandit online algorithm with a one-point sampling gradient estimator to solve the considered optimization problem. We then derive expected regret and constraint violation bounds for the proposed algorithm.

\subsection{Distributed Bandit Online Algorithm with One-Point Sampling Gradient Estimator}

\begin{algorithm}[tb]
\caption{Distributed Bandit Online Descent with One-Point Sampling Gradient Estimator}%
\label{dbco:algorithm-one}
\begin{algorithmic}[1]
\STATE \textbf{Input:}  non-increasing sequences $\{\alpha_{i,t},~\beta_{i,t},~\gamma_{i,t}\}\subseteq(0,+\infty)$, $\{\xi_{i,t}\}\subseteq(0,1)$, and $\{\delta_{i,t}\}\subseteq(0,r_i\xi_{i,t-1}],~i\in[n]$.
\STATE \textbf{Initialize:} $u_{i,1}\in\mathbb{S}^{p_i}$, $z_{i,1}\in (1-\xi_{i,1})\mathbb{X}_i$, $x_{i,1}=z_{i,1}+\delta_{i,1}u_{i,1}$, and $q_{i,1}={\bf0}_{m},~\forall i\in[n]$.
\FOR{$t=2,\dots,T$}
\FOR{$i\in[n]$ in parallel}
\STATE Select vector $u_{i,t}\in\mathbb{S}^{p_i}$ independently and uniformly at random.
\STATE Sample $f_{i,t-1}(x_{i,t-1})$ and $g_{i,t-1}(x_{i,t-1})$.
\STATE  Update
\begin{subequations}
\begin{align}
       \tilde{q}_{i,t}=&\sum_{j=1}^n[W_{t-1}]_{ij}q_{j,t-1},\label{dbco:algorithm-one:qhat}\\
       z_{i,t}=&\calP_{(1-\xi_{i,t})\mathbb{X}_i}(z_{i,t-1}-\alpha_{i,t}a_{i,t}),\label{dbco:algorithm-one:x}\\
       x_{i,t}=&z_{i,t}+\delta_{i,t}u_{i,t},\label{dbco:algorithm-one:xz}\\
       q_{i,t}=&[(1-\beta_{i,t}\gamma_{i,t})\tilde{q}_{i,t}+\gamma_{i,t}g_{i,t-1}(x_{i,t-1})]_{+}.\label{dbco:algorithm-one:q}
       \end{align}
       \end{subequations}
\STATE  Broadcast $q_{i,t}$ to $\mathcal{N}^{\outin}_i(\mathcal{G}_{t})$ and receive $[W_{t}]_{ij}q_{j,t}$ from $j\in\mathcal{N}^{\inout}_i(\mathcal{G}_{t})$.
\ENDFOR
\ENDFOR
\STATE  \textbf{Output:} $\bsx_{T}$.
\end{algorithmic}
\end{algorithm}

The proposed algorithm is given in pseudo-code as Algorithm~\ref{dbco:algorithm-one}.
In this algorithm, each agent $i$ maintains four local sequences: the local primal decision variable sequence $\{x_{i,t}\}\subseteq \mathbb{X}_i$, the local intermediate decision variable sequence $\{z_{i,t}\}\subseteq (1-\xi_{i,t})\mathbb{X}_i$, the local dual variable sequence $\{q_{i,t}\}\subseteq\mathbb{R}^m_+$, and the estimates of the average of local dual variables $\{\tilde{q}_{i,t}\}\subseteq\mathbb{R}^m_+$. They are updated recursively by the update rules (\ref{dbco:algorithm-one:qhat})--(\ref{dbco:algorithm-one:q}). In (\ref{dbco:algorithm-one:x}), $a_{i,t}$ is the updating direction information for the local intermediate decision variable defined as
\begin{align}
a_{i,t}=\hat{\nabla}_1f_{i,t-1}(z_{i,t-1})+(\hat{\nabla}_1g_{i,t-1}(z_{i,t-1}))^\top \tilde{q}_{i,t}.\label{dbco:algorithm-one:a}
\end{align}

The intuition of the update rules (\ref{dbco:algorithm-one:qhat})--(\ref{dbco:algorithm-one:q}) is as follows. The augmented Lagrangian function associated with the constrained optimization problem with cost function $f$ and constraint function $g$ is
\begin{align}\label{dbco:lagrangian}
\calA(x,\mu)=f(x)+\mu^\top g(x)-\frac{\beta}{2}\|\mu\|^2,
\end{align}
where $\{\mu\in\mathbb{R}^m_+\}$ is the Lagrange multiplier and $\beta>0$ is the regularization parameter. $\calA(x,\mu)$ is a convex-concave function and a standard primal-dual algorithm to find its saddle point is
\begin{subequations}\label{dbco:algorithm-cen}
\begin{align}
x_{k+1}=&\calP_{\mathbb{X}}(x_{k}-\alpha(\nabla f(x_k)+(\nabla g(x_k))^\top \mu_k)),\label{dbco:algorithm-cen-x}\\
\mu_{k+1}=&[\mu_k+\gamma(g(x_{k})-\beta\mu_k)]_+,\label{dbco:algorithm-cen-u}
\end{align}
\end{subequations}
where $\alpha>0$ and $\gamma>0$ are the stepsizes  used in the primal and dual updates, respectively.
The update rules (\ref{dbco:algorithm-one:qhat})--(\ref{dbco:algorithm-one:q}) are the distributed, online, and gradient-free extensions of \eqref{dbco:algorithm-cen-x} and \eqref{dbco:algorithm-cen-u}.

Algorithm~\ref{dbco:algorithm-one} generates random vectors $\tilde{q}_{i,t}$, $z_{i,t}$, $x_{i,t}$, $q_{i,t},~i\in[n],~t\in\mathbb{N}_+$. Let $\mathfrak{U}_t$ denote the $\sigma$-algebra generated by the independent and identically distributed (i.i.d.) random variables $u_{1,t},\dots,u_{n,t}$ and let $\mathcal{U}_t=\bigcup_{s=1}^{t}\mathfrak{U}_s$. It is straightforward to see that $\tilde{q}_{t+1}$, $z_{i,t}$, $x_{i,t-1}$, and $q_{i,t}$ depend on $\mathcal{U}_{t-1}$ and are independent of $\mathfrak{U}_s$ for all $s\ge t$. 

\subsection{Expected Regret and Constraint Violation Bounds}
This section states the main results on the expected regret and constraint violation bounds for Algorithm~\ref{dbco:algorithm-one}.
The following theorem characterizes these bounds based on some specially selected stepsizes, shrinkage coefficients, and exploration parameters under different feedback models and conditions. Then, a corollary is given to characterize the  expected static regret and constraint violation bounds.

\begin{theorem}\label{dbco:corollaryreg}
Suppose Assumptions~\ref{dbco:assgraph}--\ref{dbco:assfunction:function} hold. For any $T\in\mathbb{N}_+$, let $\bsx_T$ be the sequence generated by Algorithm~\ref{dbco:algorithm-one} with
\begin{align}\label{dbco:stepsize1}
&\alpha_{i,t}=\frac{r_i^2}{4mp_i^2F_{g_i}^2t^{\theta_1}},~\beta_{i,t}=\frac{2}{t^{\theta_2}},
~\gamma_{i,t}=\frac{1}{t^{1-\theta_2}},\notag\\
&\xi_{i,t}=\frac{1}{(t+1)^{\theta_3}},~\delta_{i,t}=\frac{r_i}{(t+1)^{\theta_3}},~\forall t\in\mathbb{N}_+,
\end{align} where $\theta_1\in(0,1)$, $\theta_2\in(0,\theta_1/3)$ and $\theta_3\in(\theta_2,(\theta_1-\theta_2)/2]$ are constants. Then, for any comparator sequence $\bsy_T\in\calX_{T}$,
\begin{subequations}
\begin{align}
&\mathbf{E}[\Reg(\bsx_T,\bsy_T)]\le  C_2T^{\max\{\theta_1,1-\theta_1+2\theta_3,1-\theta_3+\theta_2\}}\notag\\
&~~~~~~~~~~~~~+\max_{i\in[n]}\{\frac{8mp_i^2F_{g_i}^2R_i}{r_i^2}\}T^{\theta_1}V(\bsy_T),\label{dbco:corollaryregequ1}\\
&\mathbf{E}[\|[\sum_{t=1}^Tg_{t}(x_{t})]_+\|]
\le \sqrt{C_{3}}T^{1-\theta_2/2},\label{dbco:corollaryconsequ}
\end{align}
\end{subequations}
where $C_1=\sum_{i=1}^n(\frac{mF_gG_{g_i}(2r_i+R_{i})}{1-\theta_3+\theta_2}+\frac{G_{f_i}(2r_i+R_{i})}{1-\theta_3}
+\frac{8mp_i^2F_{g_i}^2R^2_{i}}{r_i^2}+\frac{F_{f_i}^2}{4mF_{g_i}^2(1-\theta_1+2\theta_3)}+\frac{16mp_i^2F_{g_i}^2R_i^2}{r_i^2})+\frac{C_{0}}{\theta_2}$, $C_{2}=C_{2,1}(2\sum_{i=1}^{n}F_{f_i}+C_1)$, $F_g=\max_{i\in[n]}\{F_{g_i}\}$, $C_{0}=\frac{6mn^{2}F_g^2\tau}{1-\lambda}+2mnF_g^2$, $\tau=(1-\frac{w}{2n^2})^{-2}>1$, $\lambda=(1-\frac{w}{2n^2})^{\frac{1}{\iota}}$, $C_{2,1}=2n(1+\max_{i\in[n]}\{\frac{F_{f_i}^2}{F_{g_i}^2(1-\theta_1+2\theta_3)}\}+\frac{1}{1-\theta_2})$, $w$ and $\iota$ are given in Assumption~\ref{dbco:assgraph}, $r_i$, $R_i$, $F_{f_i}$, $F_{g_i}$, $G_{f_i}$, and $G_{g_i}$ are given in Assumption~\ref{dbco:assfunction:function}; and \begin{align*}
V(\bsy_T)=\sum_{t=1}^{T-1}\sum_{i=1}^n\|y_{i,t+1}-y_{i,t}\|
\end{align*}
is the accumulated variation (path-length) of the comparator sequence $\bsy_T$.
\end{theorem}
\begin{proof}
See Appendix~\ref{dbco:corollaryregproof}.
\end{proof}
\begin{remark}\label{dbco:remarksubreg}
From \eqref{dbco:corollaryconsequ}, we see that Algorithm~\ref{dbco:algorithm-one} achieves sublinear expected constraint violation. Algorithm~\ref{dbco:algorithm-one} can also achieve
sublinear expected dynamic regret if $V(\bsx^*_T)$ grows sublinearly. In this case, there exists a constant $\nu\in[0,1)$, such that $V(\bsx^*_T)=\mathcal{O}(T^\nu)$, then setting $\bsy_T=\bsx^*_T$  and $\theta_1\in(0,1-\nu)$ in Theorem~\ref{dbco:corollaryreg}  gives $\mathbf{E}[\Reg(\bsx_T,\bsx^*_T)]=\mathbf{o}(T)$.
\end{remark}

\begin{remark}\label{dbco:remarksub}
To the best of our knowledge, Algorithm~\ref{dbco:algorithm-one} is the first algorithm to solve online convex optimization with time-varying inequality constraints in the one-point bandit feedback setting. An advantage of Algorithm~\ref{dbco:algorithm-one} is that the information about the total number of rounds is not used, which is an improvement compared to the one-point sampling algorithms in \cite{flaxman2005online,abernethy2009competing,abernethy2012interior,saha2011improved,hazan2014bandit,chen2018bandit,yuan2016online}. Note that these papers did not consider bandit feedback for time-varying inequality constraints or did not even consider time-varying inequality constraints at all. The potential drawback of Algorithm~\ref{dbco:algorithm-one} is that in order to use the sequences defined in \eqref{dbco:stepsize1}, each learner needs to know a uniform upper bound of her time-varying constraint function.
\end{remark}

Setting $\bsy_T=\check{\bsx}^*_T$ in Theorem~\ref{dbco:corollaryreg} gives the following results.

\begin{corollary}\label{dbco:corollaryreg-static}
Under the same conditions as in Theorem~\ref{dbco:corollaryreg} with $\theta_1\in(3/4,5/6]$, $\theta_2=2\theta_1-3/2$, and $\theta_3=\theta_1-1/2$, it holds that
\begin{subequations}
\begin{align}
&\mathbf{E}[\Reg(\bsx_T,\check{\bsx}^*_T)]\le  C_1T^{\theta_1},\label{dbco:corollaryregequ1-static}\\
&\mathbf{E}[\|[\sum_{t=1}^Tg_{t}(x_{t})]_+\|]
\le \sqrt{C_{2}}T^{7/4-\theta_1}.\label{dbco:corollaryconsequ-static}
\end{align}
\end{subequations}
\end{corollary}
\begin{remark}
The parameter $\theta_1$ in Corollary~\ref{dbco:corollaryreg-static} is a user-defined parameter which enables the trade-off between the expected static regret bound and the expected constraint violation bound.
Same as in \cite{flaxman2005online}, if there is no inequality constraints, i.e., $g_{i,t}\equiv{\bf 0}_m,~\forall i\in[n],~\forall t\in\mathbb{N}_+$, then by setting $\alpha_{i,t}=\frac{1}{t^{3/4}},~\beta_{i,t}=\gamma_{i,t}=0,~\xi_{i,t}=\frac{1}{(t+1)^{1/4}},~\delta_{i,t}=\frac{r_i}{(t+1)^{1/4}}$ in \eqref{dbco:stepsize1}, we have that \eqref{dbco:corollaryregequ1-static} can be replaced by $\mathbf{E}[\Reg(\bsx_T,\check{\bsx}^*_T)]\le  \hat{C}_1T^{3/4}$, where $\hat{C}_1=\sum_{i=1}^{n}(\frac{4G_{f_i}(2r_i+R_{i})}{3}
+6R^2_{i}+\frac{4p_i^2F_{f_i}^2}{3r_i^2})$. Hence,  Algorithm~\ref{dbco:algorithm-one} achieves the same expected static regret bound as the bandit algorithm in \cite{flaxman2005online}. However,  in \cite{flaxman2005online} the total number of rounds, the Lipschitz constant, and upper bound of the loss functions need to be known in advance to design the algorithm.
\end{remark}

\section{Two-Point Bandit Feedback}\label{dbco:mainresult2}
In this section, we consider the two-point bandit feedback algorithm.

\subsection{Distributed Bandit Online Algorithm with Two-Point Sampling Gradient Estimator}

\begin{algorithm}[tb]
\caption{Distributed Bandit Online Descent with Two-Point Sampling Gradient Estimator}%
\label{dbco:algorithm-two}
\begin{algorithmic}[1]
\STATE \textbf{Input:} non-increasing sequences $\{\alpha_{i,t},~\beta_{i,t},~\gamma_{i,t}\}\subseteq(0,+\infty)$, $\{\xi_{i,t}\}\subseteq(0,1)$, and $\{\delta_{i,t}\}\subseteq(0,r_i\xi_{i,t-1}],~i\in[n]$.
\STATE \textbf{Initialize:} $x_{i,1}\in (1-\xi_{i,1})\mathbb{X}_i$ and $q_{i,1}={\bf0}_{m},~\forall i\in[n]$.
\FOR{$t=2,\dots,T$}
\FOR{$i\in[n]$ in parallel}
\STATE Select vector $u_{i,t-1}\in\mathbb{S}^{p_i}$ independently and uniformly at random.
\STATE Sample $f_{i,t-1}(x_{i,t-1}+\delta_{i,t-1}u_{i,t-1})$, $f_{i,t-1}(x_{i,t-1})$, $g_{i,t-1}(x_{i,t-1}+\delta_{i,t-1}u_{i,t-1})$ and $g_{i,t-1}(x_{i,t-1})$.
\STATE  Update
\begin{subequations}
\begin{align}
       \tilde{q}_{i,t}=&\sum_{j=1}^n[W_{t-1}]_{ij}q_{j,t-1},\label{dbco:algorithm-two:qhat}\\
       x_{i,t}=&\calP_{(1-\xi_{i,t})\mathbb{X}_i}(x_{i,t-1}-\alpha_{i,t}b_{i,t}),\label{dbco:algorithm-two:x}\\
       q_{i,t}=&[(1-\gamma_{i,t}\beta_{i,t})\tilde{q}_{i,t}+\gamma_{i,t}c_{i,t}]_{+}.\label{dbco:algorithm-two:q}
       \end{align}
       \end{subequations}
\STATE  Broadcast $q_{i,t}$ to $\mathcal{N}^{\outin}_i(\mathcal{G}_{t})$ and receive $[W_{t}]_{ij}q_{j,t}$ from $j\in\mathcal{N}^{\inout}_i(\mathcal{G}_{t})$.
\ENDFOR
\ENDFOR
\STATE  \textbf{Output:} $\bsx_{T}$.
\end{algorithmic}
\end{algorithm}

With two-point bandit feedback at each round each learner can sample the values of her local loss and constraint at two points. This gives the freedom to design a more efficient algorithm which at the same time avoids the potential drawback of Algorithm~\ref{dbco:algorithm-one} as stated in Remark~\ref{dbco:remarksub} on knowing the upper bounds of the time-varying constraint functions.
The proposed algorithm is given in pseudo-code as Algorithm~\ref{dbco:algorithm-two}.
In (\ref{dbco:algorithm-two:x}), $b_{i,t}$ is the updating direction information for the local primal decision variable defined as
\begin{align}
b_{i,t}=\hat{\nabla}_2f_{i,t-1}(x_{i,t-1})+(\hat{\nabla}_2g_{i,t-1}(x_{i,t-1}))^\top \tilde{q}_{i,t}.\label{dbco:algorithm-two:a}
\end{align}
Similarly, in (\ref{dbco:algorithm-two:q}), $c_{i,t}$ is the updating direction information for the local dual variable defined as
\begin{align}
c_{i,t}=\hat{\nabla}_2 g_{i,t-1}(x_{i,t-1})(x_{i,t}-x_{i,t-1})+g_{i,t-1}(x_{i,t-1}).\label{dbco:algorithm-two:b}
\end{align}

In addition to that Algorithm~\ref{dbco:algorithm-two} uses a two-point sampling gradient estimator, another difference between Algorithms~\ref{dbco:algorithm-one} and \ref{dbco:algorithm-two} is that when updating the local dual variable, in Algorithm~\ref{dbco:algorithm-two}, $c_{i,t}$ is used to replace $g_{i,t-1}(x_{i,t-1})$. This modification is inspired by the algorithms proposed in \cite{neely2017online,yi2019distributed} and helps to avoid using other information, such as the upper bound of the loss and constraint functions.


\subsection{Expected Regret and Constraint Violation Bounds}
This section states the main results on the expected regret and constraint violation bounds for Algorithm~\ref{dbco:algorithm-two}.

\begin{theorem}\label{dbco:corollaryreg-two}
Suppose Assumptions~\ref{dbco:assgraph}--\ref{dbco:assfunction:function} hold. For any $T\in\mathbb{N}_+$, let $\bsx_T$ be the sequence generated by Algorithm~\ref{dbco:algorithm-two} with
\begin{align}\label{dbco:stepsize1-two}
&\alpha_t=\frac{1}{t^{\kappa}},~\beta_t=\frac{1}{t^{\kappa}},
~\gamma_t=\frac{1}{t^{1-\kappa}},\notag\\
&\xi_{i,t}=\frac{1}{t+1},~\delta_{i,t}=\frac{r_i}{t+1},~\forall t\in\mathbb{N}_+,
\end{align}
where $\kappa\in(0,1)$ is a constant. Then, for any comparator sequence $\bsy_T\in\calX_{T}$,
\begin{subequations}
\begin{align}
&\mathbf{E}[\Reg(\bsx_T,\bsy_T)]\le  C_3T^{\max\{\kappa,1-\kappa\}}+2R_{\max}T^{\kappa}V(\bsy_T),\label{dbco:corollaryregequ1-two}\\
&\mathbf{E}[\|[\sum_{t=1}^Tg_{t}(x_{t})]_+\|]\le \sqrt{C_{4}}T^{1-\kappa/2},\label{dbco:corollaryconsequ-two}
\end{align}
\end{subequations}
where $C_3=\sum_{i=1}^{n}(2G_{f_i}(r_i+R_{i})+8R^2_{i}+\frac{2\sqrt{m}B_1G_{g_i}R_{i}}{\kappa}+\frac{p_i^2G^2_{f_i}}{1-\kappa})+\frac{\hat{C}_{0}}{\kappa}$, $C_{4}=C_{4,1}(2\sum_{i=1}^nF_{f_i}+C_{3})$, $C_{4,1}=\sum_{i=1}^n2(\frac{2mp_i^2G_{g_i}^2+1}{1-\kappa}+1)$, $\hat{C}_{0}=\frac{6n^{2}\sqrt{m}\tau B_1F_g}{1-\lambda}+2nB_1^2$, $B_1=\sqrt{m}F_g+\sqrt{m}pG_gR_{\max}$, and $R_{\max}=\max_{i\in[n]}\{R_i\}$.
\end{theorem}
\begin{proof}
See Appendix~\ref{dbco:corollaryreg-twoproof}.
\end{proof}

\begin{remark}\label{dbco:remarksub-two}
The bounds obtained in \eqref{dbco:corollaryregequ1-two} and \eqref{dbco:corollaryconsequ-two} are the same as the bounds achieved in \cite{yi2019distributed} under the same assumptions, although \cite{yi2019distributed} considered a full-information feedback setting. In other words, in an average sense, Algorithm~\ref{dbco:algorithm-two}, which only uses two-point bandit feedback, is as efficient as the algorithm proposed in \cite{yi2019distributed}, which uses full-information feedback.
By comparing \eqref{dbco:stepsize1}, \eqref{dbco:corollaryregequ1}, and \eqref{dbco:corollaryconsequ} with \eqref{dbco:stepsize1-two}, \eqref{dbco:corollaryregequ1-two}, and \eqref{dbco:corollaryconsequ-two}, respectively, we see that if a two-point sampling gradient estimator is used then not only the uses of $F_{g_i}$, the uniform upper bound of the time-varying constraint functions, is avoided, but also the upper bounds of the expected regret and constraint violation are both reduced.
Moreover, similar to the analysis in Remark~\ref{dbco:remarksubreg}, from \eqref{dbco:corollaryconsequ-two}, we know that Algorithm~\ref{dbco:algorithm-two} achieves sublinear expected constraint violation. Algorithm~\ref{dbco:algorithm-two} can also achieve sublinear expected dynamic regret if $V(\bsx^*_T)$ grows sublinearly. In this case, there exists a constant $\nu\in[0,1)$, such that $V(\bsx^*_T)=\mathcal{O}(T^\nu)$. Then setting $\bsy_T=\bsx^*_T$  and $\kappa\in(0,1-\nu)$ in Theorem~\ref{dbco:corollaryreg-two}  gives $\mathbf{E}[\Reg(\bsx_T,\bsx^*_T)]=\mathbf{o}(T)$. An advantage of Algorithm~\ref{dbco:algorithm-two} is that the total number of rounds  or any other parameters related to loss or constraint functions are not used, which is different from the two-point sampling algorithms in \cite{agarwal2010optimal,duchi2015optimal,shamir2017optimal,yi2016tracking,shames2019online,mahdavi2012trading,chen2018bandit,cao2019online,yuan2019distributed}.
\end{remark}

Setting $\bsy_T=\check{\bsx}^*_T$ in Theorem~\ref{dbco:corollaryreg-two} gives the following results.

\begin{corollary}\label{dbco:corollaryreg-two-static}
Under the same conditions as stated in Theorem~\ref{dbco:corollaryreg-two}, it holds that
\begin{subequations}
\begin{align}
&\mathbf{E}[\Reg(\bsx_T,\check{\bsx}^*_T)]\le  C_3T^{\max\{\kappa,1-\kappa\}},\label{dbco:corollaryregequ1-two-static}\\
&\mathbf{E}[\|[\sum_{t=1}^Tg_{t}(x_{t})]_+\|]\le \sqrt{C_{4}}T^{1-\kappa/2}.\label{dbco:corollaryconsequ-two-static}
\end{align}
\end{subequations}
\end{corollary}
\begin{remark}
The parameter $\kappa$ in Corollary~\ref{dbco:corollaryreg-two-static} enables the user to trade-off the expected static regret bound for the expected constraint violation bound. For example, setting $\kappa=1/2$ in Corollary~\ref{dbco:corollaryreg-two-static} gives $\mathbf{E}[\Reg(\bsx_T,\check{\bsx}^*_T)]=\mathcal{O}(\sqrt{T})$ and $\mathbf{E}[\|[\sum_{t=1}^Tg_{t}(x_{t})]_+\|]=\mathcal{O}(T^{3/4})$. These two bounds are the same as the bounds achieved in \cite{jenatton2016adaptive,sun2017safety,mahdavi2012trading}. In other words, Algorithm~\ref{dbco:algorithm-two} is as efficient as the algorithms proposed in \cite{mahdavi2012trading,jenatton2016adaptive,sun2017safety}.  However, \cite{jenatton2016adaptive,sun2017safety} use full-information feedback and \cite{mahdavi2012trading} considers bandit setting only for the constraint functions. The algorithms proposed in \cite{mahdavi2012trading,jenatton2016adaptive,sun2017safety} are centralized and the constraint functions considered in \cite{mahdavi2012trading,jenatton2016adaptive} are time-invariant. Moreover, in \cite{mahdavi2012trading,sun2017safety} the total number of rounds and in \cite{mahdavi2012trading,jenatton2016adaptive,sun2017safety} the upper bounds of the loss and constraint functions and their subgradients need to be known in advance to design algorithms.
Also, an $\mathcal{O}(\sqrt{T})$  expected static regret bound was achieved by the bandit algorithm in \cite{shamir2017optimal}. However, in \cite{shamir2017optimal} static set constraints (rather than time-varying inequality constraints) are considered and the proposed algorithm is centralized (rather than distributed). Moreover, in \cite{shamir2017optimal} the total number of rounds and the Lipschitz constant need to be known in advance.
\end{remark}

\section{Numerical Simulations}\label{dbco:simulation}
This section evaluates the performance of Algorithms~\ref{dbco:algorithm-one} and \ref{dbco:algorithm-two} in solving the power generation example introduced in Section~\ref{dbco:example}. The local cost and constraint functions are denoted
\begin{align*}
f_{i,t}(x_{i,t})=&x_{i,t}^\top \Pi_{i,t}^\top\Pi_{i,t}x_{i,t}+\langle \pi_{i,t},x_{i,t}\rangle,\\
g_{i,t}(x_{i,t})=&x_{i,t}^\top \Phi_{i,t}^\top\Phi_{i,t}x_{i,t}+\langle \phi_{i,t},x_{i,t}\rangle+c_{i,t},
\end{align*}
where $\Pi_{i,t}\in\mathbb{R}^{p_i\times p_i}$, $\pi_{i,t}\in\mathbb{R}^{p_i}_+$, $\Phi_{i,t}\in\mathbb{R}^{p_i\times p_i}$, $\phi_{i,t}\in\mathbb{R}^{p_i}$, and $c_{i,t}\in\mathbb{R}$.  At each time $t$, an undirected graph is used as the communication graph. Specifically, connections between vertices are random and the probability of two vertices being connected is $\rho$. To guarantee that Assumption~\ref{dbco:assgraph} holds, edges $(i,i+1),~i\in[n-1]$ are added and $[W_t]_{ij}=\frac{1}{n}$ if $(j,i)\in\mathcal{E}_t$ and $[W_t]_{ii}=1-\sum_{j\in\mathcal{N}^{\inout}_i(\mathcal{G}_t)}[W_t]_{ij}$.
The parameters are set as: $n=50$, $m=1$, $p_i=6$, $\mathbb{X}_i=[-10,10]^{p_i}$, and $\rho=0.2$.  Each element of $\Pi_{i,t}$, $\pi_{i,t}$, $\Phi_{i,t}$, $\phi_{i,t}$, and $c_{i,t}$ are drawn from the discrete uniform distribution in $[-5,5]$, $[0,10]$, $[-5,5]$, $[-5,5]$, and $[-5,-1]$, respectively.

Since there are no distributed bandit online algorithms to solve the problem of distributed online optimization with time-varying coupled inequality constraints, we compare our Algorithms~\ref{dbco:algorithm-one} and \ref{dbco:algorithm-two} with the centralized one- and two-point sampling algorithms\footnote{These two algorithms use full-information feedback for constraint functions.} in \cite{chen2018bandit} and the centralized two-point sampling algorithm in \cite{cao2019online}. Figs.~\ref{dbco:figreg} and \ref{dbco:figconst} show the evolutions of $\mathbf{E}[\Reg(\bsx_T,\bsx^*_T)]/T$ and $\mathbf{E}[\|[\sum_{t=1}^Tg_{t}(x_{t})]_+\|]/T$, respectively. The average is taken over 100 realizations. Note that $\mathbf{E}[\|[\sum_{t=1}^Tg_{t}(x_{t})]_+\|]/T\rightarrow0$. This is in agreement with \eqref{dbco:corollaryconsequ}, \eqref{dbco:corollaryconsequ-two}, and the theoretical results shown in \cite{chen2018bandit,cao2019online}. From the zoomed figures,  we can see that the centralized algorithms in \cite{chen2018bandit,cao2019online} achieve smaller expected dynamic regret and constraint violation than our distributed algorithms, which is reasonable. We can also see that Algorithm~\ref{dbco:algorithm-two} achieves smaller expected dynamic regret and constraint violation than Algorithm~\ref{dbco:algorithm-one}, which is consistent with the theoretical results.

\begin{figure}
  \centering
  \includegraphics[width=\linewidth]{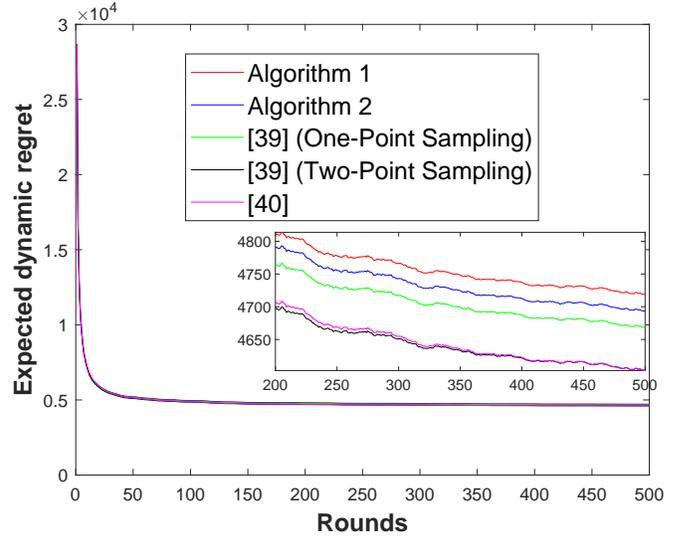}
\caption{Comparison of evolutions of the expected dynamic regret $\mathbf{E}[\Reg(\bsx_T,\bsx^*_T)]/T$.}
\label{dbco:figreg}
\end{figure}

\begin{figure}
  \centering
  \includegraphics[width=\linewidth]{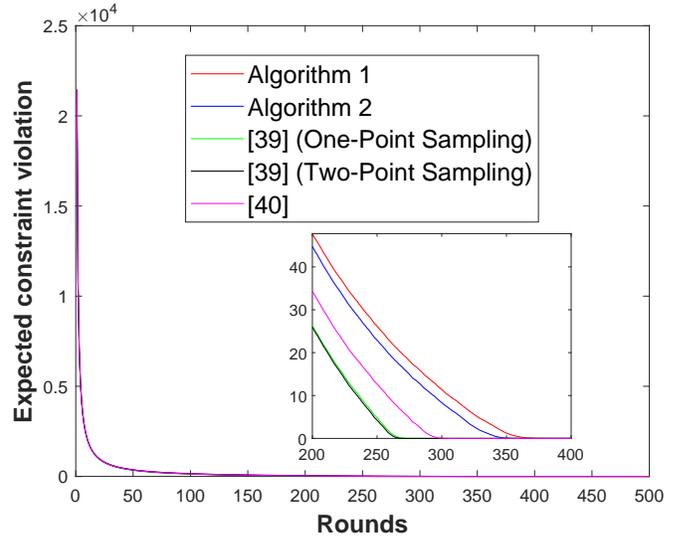}
\caption{Comparison of evolutions of the expected constraint violation $\mathbf{E}[\|[\sum_{t=1}^Tg_{t}(x_{t})]_+\|]/T$.}
\label{dbco:figconst}
\end{figure}

\section{Conclusions}\label{dbco:conclusions}
In this paper, we considered a distributed bandit online convex optimization problem with time-varying coupled inequality constraints.  We proposed distributed bandit online algorithms for one- and two-point bandit feedback. We showed that sublinear expected regret and constraint violation can be achieved. We showed that the results can be cast as non-trivial extensions of existing literature on online optimization and bandit feedback.  Future research directions include considering an adaptive choice of the number of samplings at each round by different learners. 

\bibliographystyle{IEEEtran}
\bibliography{refs}

\appendix\label{dbco:appendix}

\subsection{Useful Lemmas}

The following two lemmas are used in the proofs.
\begin{lemma}\label{dbco:lemma:projection}
Let $\mathbb{K}$ be a nonempty closed convex subset of $\mathbb{R}^{p}$ and let $a,~b,~c$ be three vectors in $\mathbb{R}^{p}$. The following statements hold.
\begin{enumerate}[label=(\alph*)]
\item For each $x\in\mathbb{R}^p$, $\calP_{\mathbb{K}}(x)$  exists and is unique.

\item $\calP_{\mathbb{K}}(x)$ is nonexpansive, i.e.,
\begin{align}\label{dbco:lemma:projection:none}
\|\calP_{\mathbb{K}}(x)-\calP_{\mathbb{K}}(y)\|\le\|x-y\|,~\forall x,y\in\mathbb{R}^p.
\end{align}
\item If $a\le b$, then
\begin{subequations}
\begin{align}
\|[a]_+\|\le&\|b\|,\label{dbco:lemma:projection:ab}\\
[a]_+\le&[b]_+.\label{dbco:lemma:projection:ab-2}
\end{align}
\end{subequations}
\item If $x_1=\calP_{\mathbb{K}}(c-a)$, then
\begin{align}
&2\langle x_1-y,a \rangle\notag\\
&\le\|y-c\|^2-\|y-x_1\|^2-\|x_1-c\|^2,~\forall y\in\mathbb{K}.\label{dbco:lemma:projection:xy}
\end{align}
\end{enumerate}
\end{lemma}
\begin{proof}
The first two parts are from Theorem~1.5.5 in \cite{facchinei2007finite}.

Substituting $x=a$ and $y=a-b$ into \eqref{dbco:lemma:projection:none} with $\mathbb{K}=\mathbb{R}^p_+$ gives \eqref{dbco:lemma:projection:ab}. If $a\le b$, then it is straightforward to see $[a]_+\le[b]_+$ since all inequalities are understood componentwise.

Denote $h(y)=\|c-y\|^2+2\langle a,y \rangle$. Then, $x_1=\argmin_{y\in\mathbb{K}}h(y)$. This optimality condition implies that
\begin{align*}
\langle x_1-y,\nabla h(x_1) \rangle\le0,~\forall y\in\mathbb{K}.
\end{align*}
Substituting $\nabla h(x_1)=2x_1-2c+2a$ into above inequality yields \eqref{dbco:lemma:projection:xy}.
\end{proof}

\begin{lemma}\label{dbco:lemma:sequencet}
For any constants $\theta\in[0,1]$, $\kappa\in[0,1)$, and $s\le T\in\mathbb{N}_+$, it holds that
\begin{subequations}
\begin{align}
&(t+1)^\kappa\left(\frac{1}{t^\theta}-\frac{1}{(t+1)^\theta}\right)\le\frac{1}{t},~\forall t\in\mathbb{N}_+,\label{dbco:sequenceupp-1}\\
&\sum_{t=s}^T\frac{1}{t^\kappa}\le\frac{T^{1-\kappa}}{1-\kappa},\label{dbco:sequenceupp}\\
&\sum_{t=s}^T\frac{1}{t}\le2\log(T),~\text{if}~T\ge3.\label{dbco:sequenceupp-2}
\end{align}
\end{subequations}
\end{lemma}
\begin{proof}
(a) Denote $h_t(\theta)=\frac{1}{t^\theta}-\frac{1}{(t+1)^\theta}$. Then, for any fixed $t\in\mathbb{N}_+$, $\max_{\theta\in[0,1]}\{h_t(\theta)\}=h_t(1)$ since $\frac{dh_t(\theta)}{d\theta}\ge0,~\forall \theta\in[0,1]$. Hence, $(t+1)^\kappa h_t(\theta)\le (t+1)^\kappa h_t(1)=\frac{(t+1)^\kappa}{t(t+1)}\le\frac{1}{t}$, i.e., \eqref{dbco:sequenceupp-1} holds.

(b) \eqref{dbco:sequenceupp} holds since
\begin{align*}
\sum_{t=s}^T\frac{1}{t^\kappa}\le\int_{s-1}^T\frac{1}{t^\kappa}dt
=\frac{T^{1-\kappa}-(s-1)^{1-\kappa}}{1-\kappa}\le\frac{T^{1-\kappa}}{1-\kappa}.
\end{align*}

(c) \eqref{dbco:sequenceupp-2} holds since
\begin{align*}
\sum_{t=s}^T\frac{1}{t}\le\frac{1}{s}+\int_s^T\frac{1}{t}dt
=\frac{1}{s}+\log(T)-\log(s)\le2\log(T).
\end{align*}
\end{proof}

\subsection{Proof of Lemma~\ref{dbco:lemma:uniformsmoothing}}\label{dbco:lemma:uniformsmoothinproof}
(a) $\nabla \hat{f}(x)=\mathbf{E}_{u\in\mathbb{S}^p}[\hat{\nabla}_1f(x)]$ is the result of Lemma~1 in \cite{flaxman2005online}. $\nabla \hat{f}(x)=\mathbf{E}_{u\in\mathbb{S}^p}[\hat{\nabla}_2f(x)]$ since $\mathbf{E}_{u\in\mathbb{S}^p}[f(x)u]=f(x)\mathbf{E}_{u\in\mathbb{S}^p}[u]={\bf 0}_p$.

(b) $(1-\xi)\mathbb{K}$ is convex since $\mathbb{K}$ is convex.

For any $x,y\in(1-\xi)\mathbb{K}$ and $\alpha\in[0,1]$, then $\alpha x+(1-\alpha)y\in(1-\xi)\mathbb{K}$ since $(1-\xi)\mathbb{K}$ is convex and $\alpha x+(1-\alpha)y+\delta v\in\mathbb{K}$ due to Lemma~\ref{dbco:lemma:flaxman}. Moreover,
\begin{align*}
&\hat{f}(\alpha x+(1-\alpha)y)\\
&=\mathbf{E}_{v\in\mathbb{B}^p}[f(\alpha x+(1-\alpha)y+\delta v)]\\
&\le \mathbf{E}_{v\in\mathbb{B}^p}[\alpha f(x+\delta v)+(1-\alpha) f(y+\delta v)]\\
&=\alpha\hat{f}(x)+(1-\alpha)\hat{f}(y).
\end{align*}
Hence, $\hat{f}$ is convex on $(1-\xi)\mathbb{K}$.

From Lemma~\ref{dbco:lemma:flaxman}, we know that $(1-\xi)\mathbb{K}$ is a subset of the interior of $\mathbb{K}$. Then, for any $x\in(1-\xi)\mathbb{K}$, from Theorem~3.1.15 in \cite{nesterov2018lectures}, we know that $\nabla f(x)$ exists. Moreover,
\begin{align*}
&\hat{f}(x)=\mathbf{E}_{v\in\mathbb{B}^p}[f(x+\delta v)]\\
&\ge \mathbf{E}_{v\in\mathbb{B}^p}[f(x)+\delta\langle\nabla f(x),v\rangle]=f(x).
\end{align*}

(c) For any $x,y\in(1-\xi)\mathbb{K}$,
\begin{align*}
&|\hat{f}(x)-\hat{f}(y)|=|\mathbf{E}_{v\in\mathbb{B}^p}[f(x+\delta v)-f(y+\delta v)]|\\
&\le \mathbf{E}_{v\in\mathbb{B}^p}[|f(x+\delta v)-f(y+\delta v)|]\\
&\le \mathbf{E}_{v\in\mathbb{B}^p}[L_0(f)\|x-y\|]=L_0(f)\|x-y\|.
\end{align*}
Hence, $\hat{f}$ is Lipschitz-continuous on $(1-\xi)\mathbb{K}$ with constant $L_0(f)$.

Similarly,
\begin{align*}
&\|\nabla\hat{f}(x)-\nabla\hat{f}(y)\|=\frac{p}{\delta}\|\mathbf{E}_{u\in\mathbb{S}^p}[f(x+\delta u)u-f(y+\delta u)u]\|\\
&\le \frac{p}{\delta}\mathbf{E}_{u\in\mathbb{S}^p}[|f(x+\delta u)-f(y+\delta u)|\|u\|]\\
&\le \frac{p}{\delta}\mathbf{E}_{u\in\mathbb{S}^p}[L_0(f)\|x-y\|]=\frac{pL_0(f)}{\delta}\|x-y\|.
\end{align*}
Hence, $\nabla\hat{f}$ is Lipschitz-continuous on $(1-\xi)\mathbb{K}$ with constant $pL_0(f)/\delta$.

For any $x\in(1-\xi)\mathbb{K}$,
\begin{align*}
&|\hat{f}(x)-f(x)|=|\mathbf{E}_{v\in\mathbb{B}^p}[f(x+\delta v)]-\mathbf{E}_{v\in\mathbb{B}^p}[f(x)]|\\
&\le \mathbf{E}_{v\in\mathbb{B}^p}[|f(x+\delta v)-f(x)|]\\
&\le \mathbf{E}_{v\in\mathbb{B}^p}[\delta L_0(f)\|v\|]\le \mathbf{E}_{v\in\mathbb{B}^p}[\delta L_0(f)]=\delta L_0(f).
\end{align*}

(d) For any $x\in(1-\xi)\mathbb{K}$ and $u\in\mathbb{S}^p$,
\begin{align*}
&|\hat{f}(x)|=|\mathbf{E}_{v\in\mathbb{B}^p}[f(x+\delta v)]|\\
&\le \mathbf{E}_{v\in\mathbb{B}^p}[|f(x+\delta v)|]\le F_0(f),
\end{align*}
and
\begin{align*}
&\|\hat{\nabla}_1f(x)\|=\|\frac{p}{\delta}f(x+\delta u)u\|\\
&\le \frac{p}{\delta}|f(x+\delta u)|\|u\|\le\frac{pF_0(f)}{\delta}.
\end{align*}

(e) For any $x\in(1-\xi)\mathbb{K}$ and $u\in\mathbb{S}^p$,
\begin{align*}
&\|\hat{\nabla}_2f(x)\|=\|\frac{p}{\delta}(f(x+\delta u)-f(x))u\|\\
&\le\frac{pL_0(f)}{\delta}\|x+\delta u-x\|\|u\|=pL_0(f).
\end{align*}

\subsection{Proof of Theorem~\ref{dbco:corollaryreg}}\label{dbco:corollaryregproof}

To prove Theorem~\ref{dbco:corollaryreg}, the following three lemmas are used.
Lemma~\ref{dbco:lemma_virtualbound} presents the results on the local dual variables, while Lemma~\ref{dbco:lemma_regretdelta} provides an upper bound for the regret of one round.
Lemma~\ref{dbco:theoremreg} provides the expected regret constraint violation bounds for Algorithm~\ref{dbco:algorithm-one} for the general case.

To simplify notation, we denote $\beta_t=\beta_{i,t}$, $\gamma_t=\gamma_{i,t}$, and $\xi_t=\xi_{i,t}$.

\begin{lemma}\label{dbco:lemma_virtualbound}
Suppose Assumptions \ref{dbco:assgraph}--\ref{dbco:assfunction:function} hold. For all $i\in[n]$ and $t\in\mathbb{N}_+$, $\tilde{q}_{i,t}$ and $q_{i,t}$ generated by Algorithm~\ref{dbco:algorithm-one} satisfy
\begin{subequations}
\begin{align}
&\|\tilde{q}_{i,t+1}\|\le \frac{\sqrt{m}F_g}{\beta_t},~\|q_{i,t}\|\le \frac{\sqrt{m}F_g}{\beta_t},\label{dbco:lemma_virtualboundeqy}\\
&\|\tilde{q}_{i,t+1}-\bar{q}_{t}\|
\le 2\sqrt{m}nF_g\tau \sum_{s=1}^{t-1}\gamma_{s+1}\lambda^{t-1-s},\label{dbco:lemma_qbarequ}\\
&\frac{\Delta_{t+1}}{2\gamma_{t+1}}\nonumber\\
&\le (\bar{q}_{t}-q)^\top g_{t}(x_{t})+2mnF_g^2\gamma_{t+1}
+\frac{n\beta_{t+1}}{2}\|q\|^2+d_{1}(t)
,\label{dbco:gvirtualnorm}
\end{align}
\end{subequations}
where  $\bar{q}_{t}=\frac{1}{n}\sum_{i=1}^nq_{i,t}$,
\begin{align}\label{dbco:delta}\Delta_{t}=\sum_{i=1}^n\|q_{i,t}-q\|^2
-(1-\beta_t\gamma_t)\sum_{i=1}^n\|q_{i,t-1}-q\|^2,
\end{align}
$q$ is an arbitrary vector in $\mathbb{R}^m_+$,
\begin{align*}
d_{1}(t)=2mn^2F_g^2\tau \sum_{s=1}^{t}\gamma_{s+1}\lambda^{t-s}.
\end{align*}
\end{lemma}
\begin{proof}

(a) From (\ref{dbco:assfunction:ftgtupper}), we have
\begin{align}\label{dbco:lemma_qbarequb}
\|g_{i,t}(x_{i,t})\|\le \sqrt{m}F_g,~\forall i\in[n],~\forall t\in\mathbb{N}_+.
\end{align}

We prove (\ref{dbco:lemma_virtualboundeqy}) by induction.

It is straightforward to see that $q_{i,1}=\tilde{q}_{i,2}={\bf 0}_{m},~\forall i\in[n]$, thus $\|\tilde{q}_{i,2}\|\le \frac{\sqrt{m}F_g}{\beta_1},~\|q_{i,1}\|\le \frac{\sqrt{m}F_g}{\beta_1},~\forall i\in[n]$. Assume that (\ref{dbco:lemma_virtualboundeqy}) is true at time $t$ for all $i\in[n]$. We show that it remains true at time $t+1$.
Firstly, from \eqref{dbco:lemma:projection:ab}, (\ref{dbco:algorithm-one:q}), (\ref{dbco:lemma_qbarequb}), $1-\gamma_{t+1}\beta_{t+1}\ge0$, and $\beta_t\ge\beta_{t+1}$ we know that for all $i\in[n]$,
\begin{align*}
&\|q_{i,t+1}\|\le(1-\gamma_{t+1}\beta_{t+1})\|\tilde{q}_{i,t+1}\|+\gamma_{t+1}\|g_{i,t}(x_{i,t})\|\\
&\le(1-\gamma_{t+1}\beta_{t+1})\frac{\sqrt{m}F_g}{\beta_{t}}+\gamma_{t+1} \sqrt{m}F_g\\
&\le(1-\gamma_{t+1}\beta_{t+1})\frac{\sqrt{m}F_g}{\beta_{t+1}}+\gamma_{t+1} \sqrt{m}F_g\le \frac{ \sqrt{m}F_g}{\beta_{t+1}}.
\end{align*}
Then, the convexity of norms and $\sum_{j=1}^n[W_{t}]_{ij}=1$ yield
\begin{align*}
\|\tilde{q}_{i,t+2}\|\le&\sum_{j=1}^n[W_{t+1}]_{ij}\|q_{j,t+1}\|\le\sum_{j=1}^n[W_{t}]_{ij}\frac{ \sqrt{m}F_g}{\beta_{t+1}}\\
=& \frac{\sqrt{m}F_g}{\beta_{t+1}},~\forall i\in[n].
\end{align*}
Thus, (\ref{dbco:lemma_virtualboundeqy}) follows.

(b) Note that
(\ref{dbco:algorithm-one:q}) can be rewritten as
\begin{align}\label{dbco:lemma_qbarequeps-2}
q_{i,t+1}=\sum_{j=1}^n[W_{t}]_{ij}q_{j,t}+\epsilon^q_{i,t},
\end{align}
where $\epsilon^q_{i,t}=[(1-\gamma_{t+1}\beta_{t+1})\tilde{q}_{i,t+1}+\gamma_{t+1}g_{i,t}(x_{i,t})]_{+}-\tilde{q}_{i,t+1}$.
Then, (\ref{dbco:lemma:projection:none}), (\ref{dbco:lemma_virtualboundeqy}), and (\ref{dbco:lemma_qbarequb}) give
\begin{align}\label{dbco:lemma_qbarequeps}
\|\epsilon^q_{i,t}\|
\le&\|-\gamma_{t+1}\beta_{t+1}\tilde{q}_{i,t+1}+\gamma_{t+1}g_{i,t}(x_{i,t})\|\nonumber\\
\le&2\sqrt{m}F_g\gamma_{t+1},~\forall i\in[n].
\end{align}
Then, from Assumption~\ref{dbco:assgraph}, Lemma 2 in \cite{lee2017sublinear}, $q_{i,1}={\bf 0}_{m},~\forall i\in[n]$, and (\ref{dbco:lemma_qbarequeps}), we know that for any $i\in[n]$ and $ t\in\mathbb{N}_+$,
\begin{align}\label{dbco:lemma_qbarequeps-1}
\|q_{i,t+1}-\bar{q}_{t+1}\|
\le 2\sqrt{m}nF_g\tau \sum_{s=1}^{t}\gamma_{s+1}\lambda^{t-s}.
\end{align}
So (\ref{dbco:lemma_qbarequ}) follows since $\sum_{j=1}^n[W_{t}]_{ij}=1$ and $\| \tilde{q}_{i,t+1}-\bar{q}_{t}\|=\|\sum_{j=1}^n[W_{t}]_{ij}q_{j,t}-\bar{q}_{t}\|\le \sum_{j=1}^n[W_{t}]_{ij}\|q_{j,t}-\bar{q}_{t}\|$.

(c) Applying (\ref{dbco:lemma:projection:none}) to (\ref{dbco:algorithm-one:q}) yields
\begin{align}
&\|q_{i,t}-q\|^2\nonumber\\
&\le\|(1-\beta_t\gamma_t)\tilde{q}_{i,t}+\gamma_tg_{i,t-1}(x_{i,t-1})-q\|^2\nonumber\\
&=\|\tilde{q}_{i,t}-q\|^2+\gamma_t^2\|g_{i,t-1}(x_{i,t-1})-\beta_t\tilde{q}_{i,t}\|^2\nonumber\\
&~~~+2\gamma_t[\tilde{q}_{i,t}-q]^\top g_{i,t-1}(x_{i,t-1})-2\beta_t\gamma_t[\tilde{q}_{i,t}-q]^\top\tilde{q}_{i,t}.\label{dbco:qmu}
\end{align}
For the first term of the right-hand side of (\ref{dbco:qmu}), by convexity of norms and $\sum_{j=1}^n[W_{t-1}]_{ij}=1$, it can be concluded that
\begin{align}
\|\tilde{q}_{i,t}-q\|^2=&\|\sum_{j=1}^n[W_{t-1}]_{ij}q_{j,t-1}-\sum_{j=1}^n[W_{t-1}]_{ij}q\|^2\nonumber\\
\le&\sum_{j=1}^n[W_{t-1}]_{ij}\|q_{j,t-1}-q\|^2.\label{dbco:qmu0}
\end{align}
For the second term of the right-hand side of (\ref{dbco:qmu}), (\ref{dbco:lemma_virtualboundeqy}) and (\ref{dbco:lemma_qbarequb})  yield
\begin{align}
\gamma_t^2\|g_{i,t-1}(x_{i,t-1})-\beta_t\tilde{q}_{i,t}\|^2
\le(2\sqrt{m}F_g\gamma_t)^2.\label{dbco:qmu1}
\end{align}
For the fourth term of the right-hand side of (\ref{dbco:qmu}),  we have
\begin{align}
&2\gamma_t[\tilde{q}_{i,t}-q]^\top g_{i,t-1}(x_{i,t-1})
=2\gamma_t[\bar{q}_{t-1}-q]^\top g_{i,t-1}(x_{i,t-1})\nonumber\\
&+2\gamma_t[\tilde{q}_{i,t}-\bar{q}_{t-1}]^\top g_{i,t-1}(x_{i,t-1}).\label{dbco:qmu4}
\end{align}
Moreover, from (\ref{dbco:lemma_qbarequb}) and (\ref{dbco:lemma_qbarequ}), we have
\begin{align}
&2\gamma_t[\tilde{q}_{i,t}-\bar{q}_{t-1}]^\top g_{i,t-1}(x_{i,t-1})\nonumber\\
&\le2\gamma_t\|\tilde{q}_{i,t}-\bar{q}_{t-1}\|\|g_{i,t-1}(x_{i,t-1})\|
\le\frac{2\gamma_td_{1}(t-1)}{n}.\label{dbco:qmu5}
\end{align}
For the last term of the right-hand side of (\ref{dbco:qmu}),  neglecting the nonnegative term $\beta_t\gamma_t\|\tilde{q}_{i,t}\|^2$ gives
\begin{align}
-2\beta_t\gamma_t[\tilde{q}_{i,t}-q]^\top\tilde{q}_{i,t}
\le\beta_t\gamma_t(\|q\|^2-\|\tilde{q}_{i,t}-q\|^2).\label{dbco:qmu3}
\end{align}
 Combining (\ref{dbco:qmu})--(\ref{dbco:qmu3}), summing over $i\in[n]$, dividing by $2\gamma_t$, using $\sum_{i=1}^n[W_{t-1}]_{ij}=1,~\forall t\in\mathbb{N}_+$,  setting $t=t+1$, and rearranging the terms yields (\ref{dbco:gvirtualnorm}).
\end{proof}

\begin{lemma}\label{dbco:lemma_regretdelta}
Suppose Assumptions \ref{dbco:assgraph}--\ref{dbco:assfunction:function} hold. For all $i\in[n]$, let $\{x_{t}\}$ be the sequence generated by Algorithm~\ref{dbco:algorithm-one} and $\{y_t\}$ be an arbitrary sequence in $\mathbb{X}$, then
\begin{align}
&f_{t}(x_{t})-f_{t}(y_{t})\nonumber\\
&\le (\bar{q}_{t})^\top (g_{t}(y_{t})- g_{t}(x_{t}))+2d_{1}(t)+d_2(t)\nonumber\\
&~~~+\sum_{i=1}^n\frac{p_i^2F_{f_i}^2\alpha_{i,t+1}}{\delta_{i,t}^2}
+\sum_{i=1}^n\frac{2R_{i}\|y_{i,t+1}-y_{i,t}\|}{\alpha_{i,t+1}}\nonumber\\
&~~~+d_{3}(t)+\mathbf{E}_{\mathfrak{U}_{t}}[d_{4}(t)],~\forall t\in\mathbb{N}_+,\label{dbco:lemma_regretdeltaequ}
\end{align}
where $d_1(t)$ is given in Lemma~\ref{dbco:lemma_virtualbound},
\begin{align*}
d_2(t)=&\sum_{i=1}^n\Big\{(2\delta_{i,t}+R_{i}\xi_{t})
(\sqrt{m}G_{g_i}\|q_{i,t}\|+G_{f_i})\\
&~~~~~~+\frac{2R^2_{i}(\xi_{t}-\xi_{t+1})}{\alpha_{i,t+1}}\Big\},\\
d_{3}(t)=&2m\max_{i\in[n]}\{\frac{p_i^2F_{g_i}^2\alpha_{i,t+1}}{\delta_{i,t}^2}\}(n\|q\|^2
+\sum_{i=1}^n\|q_{i,t}-q\|^2),\\
d_{4}(t)=&\sum_{i=1}^n\frac{\|\check{y}_{i,t}-z_{i,t}\|^2
-\|\check{y}_{i,t+1}-z_{i,t+1}\|^2}{2\alpha_{i,t+1}},
\end{align*}
and $\check{y}_{i,t}=(1-\xi_{t})y_{i,t}$.
\end{lemma}
\begin{proof} (a) For any $i\in[n]$, $t\in\mathbb{N}_+$ and $x\in(1-\xi_{t})\mathbb{X}_i$, denote
\begin{align*}
\hat{f}_{i,t}(x)&=\mathbf{E}_{v\in\mathbb{B}^p}[f_{i,t}(x+\delta_{i,t} v)],\\
\hat{g}_{i,t}(x)&=\mathbf{E}_{v\in\mathbb{B}^p}[g_{i,t}(x+\delta_{i,t} v)].
\end{align*}

From Lemma~\ref{dbco:lemma:uniformsmoothing}, (\ref{dbco:assfunction:ftgtupper}), \eqref{dbco:lemma_qbarequb}, \eqref{dbco:assfunction:functionLipf}, and \eqref{dbco:assfunction:functionLipg}, we know that $\hat{f}_{i,t}(x)$ and $\hat{g}_{i,t}(x)$ are convex on $(1-\xi_{t})\mathbb{X}_i$, and for any $i\in[n]$, $t\in\mathbb{N}_+$ and $x\in(1-\xi_{t})\mathbb{X}_i$,
\begin{subequations}
\begin{align}
&\nabla\hat{f}_{i,t}(x)=\mathbf{E}_{\mathfrak{U}_{t}}[\hat{\nabla}_1f_{i,t}(x)],
\label{dbco:lemma_regretdeltaequ:fsmooth1}\\
&f_{i,t}(x)\le \hat{f}_{i,t}(x)\le f_{i,t}(x)+G_{f_i}\delta_{i,t} ,\label{dbco:lemma_regretdeltaequ:fsmooth2}\\
&\|\hat{\nabla}_1f_{i,t}(x)\|\le \frac{p_iF_{f_i}}{\delta_{i,t}},
\label{dbco:lemma_regretdeltaequ:fsmooth3}\\
&\nabla\hat{g}_{i,t}(x)=\mathbf{E}_{\mathfrak{U}_{t}}[\hat{\nabla}_1g_{i,t}(x)],
\label{dbco:lemma_regretdeltaequ:gsmooth1}\\
&g_{i,t}(x)\le \hat{g}_{i,t}(x)\le g_{i,t}(x)+ G_{g_i}\delta_{i,t}{\bf 1}_m,\label{dbco:lemma_regretdeltaequ:gsmooth2}\\
&\|\hat{\nabla}_1g_{i,t}(x)\|\le \frac{\sqrt{m}p_iF_{g_i}}{\delta_{i,t}},
\label{dbco:lemma_regretdeltaequ:gsmooth4}\\
&\|\hat{g}_{i,t}(x)\|\le \sqrt{m}F_{g_i}.\label{dbco:lemma_regretdeltaequ:gsmooth3}
\end{align}
\end{subequations}
Then, \eqref{dbco:assfunction:functionLipf}, \eqref{dbco:assfunction:functionLipg}, \eqref{dbco:domainupper}, and \eqref{dbco:lemma_regretdeltaequ:fsmooth2} yield
\begin{subequations}
\begin{align}
&|f_{i,t}(x_{i,t})-f_{i,t}(z_{i,t})|
\le G_{f_i}\|x_{i,t}-z_{i,t}\|\le G_{f_i}\delta_{i,t},\label{dbco:lemma_regretdelta:equ2.2}\\
&\|g_{i,t}(x_{i,t})-g_{i,t}(z_{i,t})\|\nonumber\\
&\le \sqrt{m}G_{g_i}\|x_{i,t}-z_{i,t}\|\le\sqrt{m}G_{g_i}\delta_{i,t},\label{dbco:lemma_regretdelta:equ2.3}\\
&\hat{f}_{i,t}(\check{y}_{i,t})-f_{i,t}(y_{i,t})\nonumber\\
&=f_{i,t}(\check{y}_{i,t})-f_{i,t}(y_{i,t})
+\hat{f}_{i,t}(\check{y}_{i,t})-f_{i,t}(\check{y}_{i,t})\nonumber\\
&\le G_{f_i}\|\check{y}_{i,t}-y_{i,t}\|+\hat{f}_{i,t}(\check{y}_{i,t})-f_{i,t}(\check{y}_{i,t})\nonumber\\
&\le G_{f_i}R_{i}\xi_{t}+ G_{f_i}\delta_{i,t},\label{dbco:lemma_regretdelta:equ1}\\
&f_{i,t}(z_{i,t})-\hat{f}_{i,t}(z_{i,t})
\le0,\label{dbco:lemma_regretdelta:equ2.1}\\
&\|g_{i,t}(\check{y}_{i,t})-g_{i,t}(y_{i,t})\|
\le \sqrt{m}G_{g_i}R_{i}\xi_{t}.\label{dbco:lemma_regretdelta:equ2}
\end{align}
\end{subequations}

From $\hat{f}_{i,t}(x)$ is convex on $(1-\xi_{t})\mathbb{X}_i$, we have that
\begin{align}\label{dbco:lemma_regretdelta:equ3}
&\hat{f}_{i,t}(z_{i,t})-\hat{f}_{i,t}(\check{y}_{i,t})
\le\langle\nabla\hat{f}_{i,t}(z_{i,t}),z_{i,t}-\check{y}_{i,t}\rangle\notag\\
&=\langle \mathbf{E}_{\mathfrak{U}_{t}}[\hat{\nabla}_1f_{i,t}(z_{i,t})],z_{i,t}-\check{y}_{i,t}\rangle\notag\\
&=\mathbf{E}_{\mathfrak{U}_{t}}[\langle \hat{\nabla}_1f_{i,t}(z_{i,t}),z_{i,t}-\check{y}_{i,t}\rangle],
\end{align}
where the first equality holds from \eqref{dbco:lemma_regretdeltaequ:fsmooth1} and the last equality holds since $z_{i,t}$ is independent of $\mathfrak{U}_{t}$.

Next, we rewrite the right-hand side of \eqref{dbco:lemma_regretdelta:equ3} into two terms and bound them individually.
\begin{align}\label{dbco:lemma_regretdelta:equ4}
&\mathbf{E}_{\mathfrak{U}_{t}}[\langle \hat{\nabla}_1f_{i,t}(z_{i,t}),z_{i,t}-\check{y}_{i,t}\rangle]\notag\\
&=\mathbf{E}_{\mathfrak{U}_{t}}[\langle\hat{\nabla}_1f_{i,t}(z_{i,t}),z_{i,t}-z_{i,t+1}\rangle]\notag\\
&~~~+\mathbf{E}_{\mathfrak{U}_{t}}[\langle\hat{\nabla}_1f_{i,t}(z_{i,t}),z_{i,t+1}-\check{y}_{i,t}\rangle].
\end{align}
For the first term of the right-hand side of \eqref{dbco:lemma_regretdelta:equ4}, the Cauchy-Schwarz  inequality and (\ref{dbco:lemma_regretdeltaequ:fsmooth3}) give
\begin{align}\label{dbco:lemma_regretdelta:equ5}
&\langle\hat{\nabla}_1f_{i,t}(z_{i,t}),z_{i,t}-z_{i,t+1}\rangle\notag\\
&\le \|\hat{\nabla}_1f_{i,t}(z_{i,t})\|\|z_{i,t}-z_{i,t+1}\|\le
\frac{p_iF_{f_i}}{\delta_{i,t}}\|z_{i,t}-z_{i,t+1}\|\notag\\
&\le\frac{p_i^2F_{f_i}^2\alpha_{i,t+1}}{\delta_{i,t}^2}+ \frac{1}{4\alpha_{i,t+1}}\|z_{i,t}-z_{i,t+1}\|^2.
\end{align}
For the second term of the right-hand side of \eqref{dbco:lemma_regretdelta:equ4}, it follows from \eqref{dbco:algorithm-one:a} that
\begin{align}
&\mathbf{E}_{\mathfrak{U}_{t}}[\langle\hat{\nabla}_1 f_{i,t}(z_{i,t}),z_{i,t+1}-\check{y}_{i,t}\rangle]\nonumber\\
&=\mathbf{E}_{\mathfrak{U}_{t}}[\langle(\hat{\nabla}_1g_{i,t}(z_{i,t}))^\top \tilde{q}_{i,t+1},\check{y}_{i,t}-z_{i,t+1}\rangle]\nonumber\\
&~~~+\mathbf{E}_{\mathfrak{U}_{t}}[\langle a_{i,t+1},z_{i,t+1}-\check{y}_{i,t}\rangle]\nonumber\\
&=\mathbf{E}_{\mathfrak{U}_{t}}[\langle(\hat{\nabla}_1g_{i,t}(z_{i,t}))^\top  \tilde{q}_{i,t+1},\check{y}_{i,t}-z_{i,t}\rangle]\nonumber\\
&~~~+\mathbf{E}_{\mathfrak{U}_{t}}[\langle(\hat{\nabla}_1g_{i,t}(z_{i,t}))^\top  \tilde{q}_{i,t+1},z_{i,t}-z_{i,t+1}\rangle]\nonumber\\
&~~~+\mathbf{E}_{\mathfrak{U}_{t}}[\langle a_{i,t+1},z_{i,t+1}-\check{y}_{i,t}\rangle].\label{dbco:lemma_regretdelta:equ6}
\end{align}
For the first term of the right-hand side of \eqref{dbco:lemma_regretdelta:equ6}, noting that $x_{i,t}$ and $\tilde{q}_{i,t+1}$ are dependent of $\mathfrak{U}_{t}$, from \eqref{dbco:lemma_regretdeltaequ:gsmooth1}, $\tilde{q}_{i,t+1}\ge{\bf0}_{m}$, $\bar{q}_{t}\ge{\bf0}_{m}$, \eqref{dbco:lemma_regretdeltaequ:gsmooth2}, and that $\hat{g}_{i,t}$ is convex, we have
\begin{align}
&\mathbf{E}_{\mathfrak{U}_{t}}[\langle(\hat{\nabla}_1g_{i,t}(z_{i,t}))^\top  \tilde{q}_{i,t+1},\check{y}_{i,t}-z_{i,t}\rangle]\nonumber\\
&=\langle(\mathbf{E}_{\mathfrak{U}_{t}}[\hat{\nabla}_1g_{i,t}(z_{i,t})])^\top  \tilde{q}_{i,t+1},\check{y}_{i,t}-z_{i,t}\rangle\nonumber\\
&=\langle(\nabla\hat{g}_{i,t}(z_{i,t}))^\top  \tilde{q}_{i,t+1},\check{y}_{i,t}-z_{i,t}\rangle\nonumber\\
&\le[ \tilde{q}_{i,t+1}]^\top \hat{g}_{i,t}(\check{y}_{{i,t}})-[ \tilde{q}_{i,t+1}]^\top \hat{g}_{i,t}(z_{i,t})\nonumber\\
&=[\bar{q}_{t}]^\top [\hat{g}_{i,t}(\check{y}_{i,t})- \hat{g}_{i,t}(z_{i,t})]\nonumber\\
&~~~+[ \tilde{q}_{i,t+1}-\bar{q}_{t}]^\top [\hat{g}_{i,t}(\check{y}_{i,t})- \hat{g}_{i,t}(z_{i,t})]\nonumber\\
&\le[\bar{q}_{t}]^\top [g_{i,t}(\check{y}_{i,t})+\delta_{i,t} G_{g_i}{\bf 1}_m- g_{i,t}(z_{i,t})]\nonumber\\
&~~~+[ \tilde{q}_{i,t+1}-\bar{q}_{t}]^\top [\hat{g}_{i,t}(\check{y}_{i,t})- \hat{g}_{i,t}(z_{i,t})].\label{dbco:lemma_regretdelta:equ8}
\end{align}
From \eqref{dbco:lemma_qbarequ} and (\ref{dbco:lemma_regretdeltaequ:gsmooth3}), we have
\begin{align}\label{dbco:lemma_regretdelta:equ9}
[\tilde{q}_{i,t+1}-\bar{q}_{t}]^\top [\hat{g}_{i,t}(\check{y}_{i,t})- \hat{g}_{i,t}(z_{i,t})]
\le
\frac{2d_{1}(t)}{n}.
\end{align}
For the second term of the right-hand side of \eqref{dbco:lemma_regretdelta:equ6}, from the Cauchy-Schwarz  inequality, \eqref{dbco:lemma_regretdeltaequ:gsmooth4}, and \eqref{dbco:qmu0} we have
\begin{align}
&\langle(\hat{\nabla}_1 g_{i,t}(z_{i,t}))^\top  \tilde{q}_{i,t+1},z_{i,t}-z_{i,t+1}\rangle\nonumber\\
&=q^\top\hat{\nabla}_1 g_{i,t}(z_{i,t})(z_{i,t}-z_{i,t+1})\nonumber\\
&~~~+(\tilde{q}_{i,t+1}-q)^\top\hat{\nabla}_1 g_{i,t}(z_{i,t})(z_{i,t}-z_{i,t+1})\nonumber\\
&\le \frac{2mp_i^2F_{g_i}^2\alpha_{i,t+1}}{\delta_{i,t}^2}\|q\|^2
+\frac{1}{8\alpha_{i,t+1}}\|z_{i,t+1}-z_{i,t}\|^2\nonumber\\
&~~~+\frac{2mp_i^2F_{g_i}^2\alpha_{i,t+1}}{\delta_{i,t}^2}\|\tilde{q}_{i,t+1}-q\|^2
+\frac{1}{8\alpha_{i,t+1}}\|z_{i,t+1}-z_{i,t}\|^2\nonumber\\
&\le 2m\max_{i\in[n]}\{\frac{p_i^2F_{g_i}^2\alpha_{i,t+1}}{\delta_{i,t}^2}\}\|q\|^2
+\frac{1}{4\alpha_{i,t+1}}\|z_{i,t+1}-z_{i,t}\|^2\nonumber\\
&~~~+2m\max_{i\in[n]}\{\frac{p_i^2F_{g_i}^2\alpha_{i,t+1}}{\delta_{i,t}^2}\}\sum_{j=1}^n[W_{t}]_{ij}\|q_{j,t}-q\|^2.\label{dbco:qmu2}
\end{align}
For the last term of the right-hand side of \eqref{dbco:lemma_regretdelta:equ6}, noting that $\check{y}_{i,t}\in(1-\xi_{t})\mathbb{X}_i\subseteq(1-\xi_{t+1})\mathbb{X}_i$ since $\xi_{t}\ge\xi_{t+1}$ and applying (\ref{dbco:lemma:projection:xy}) to the update rule (\ref{dbco:algorithm-one:x}) yields
\begin{align}\label{dbco:lemma_regretdelta:equ10}
&2\alpha_{i,t+1}\langle a_{i,t+1},z_{i,t+1}-\check{y}_{i,t}\rangle\nonumber\\
&\le\|\check{y}_{i,t}-z_{i,t}\|^2-\|\check{y}_{i,t}-z_{i,t+1}\|^2
-\|z_{i,t+1}-z_{i,t}\|^2\nonumber\\
&=\|\check{y}_{i,t+1}-z_{i,t+1}\|^2-\|\check{y}_{i,t}-z_{i,t+1}\|^2
+\|\check{y}_{i,t}-z_{i,t}\|^2\nonumber\\
&~~~-\|\check{y}_{i,t+1}-z_{i,t+1}\|^2-\|z_{i,t+1}-z_{i,t}\|^2.
\end{align}
The first two terms of the right-hand side of \eqref{dbco:lemma_regretdelta:equ10} can be bounded by
\begin{align}\label{dbco:lemma_regretdelta:equ11}
&\|\check{y}_{i,t+1}-z_{i,t+1}\|^2-\|\check{y}_{i,t}-z_{i,t+1}\|^2\notag\\
&\le\|\check{y}_{i,t+1}-\check{y}_{i,t}\|\|\check{y}_{i,t+1}+\check{y}_{i,t}-2z_{i,t+1}\|\notag\\
&\le 4R_{i}\|(1-\xi_{t+1})y_{i,t+1}-(1-\xi_{t})y_{i,t}\|\notag\\
&=4R_{i}\|(1-\xi_{t+1})(y_{i,t+1}-y_{i,t})+(\xi_{t}-\xi_{t+1})y_{i,t}\|\notag\\
&\le4R_{i}\|y_{i,t+1}-y_{i,t}\|+4R^2_{i}(\xi_{t}-\xi_{t+1}),
\end{align}
where the last inequality holds since $\{\xi_{t}\}\subseteq(0,1)$ is non-increasing.

Combining (\ref{dbco:lemma_regretdelta:equ1})--(\ref{dbco:lemma_regretdelta:equ11}), taking expectation in $\mathfrak{U}_{t+1}$, summing over $i\in[n]$, and rearranging the terms yields (\ref{dbco:lemma_regretdeltaequ}).
\end{proof}

\begin{lemma}\label{dbco:theoremreg}
Suppose Assumptions~\ref{dbco:assgraph}--\ref{dbco:assfunction:function} hold. For any $T\in\mathbb{N}_+$, let $\bsx_T$ be the sequence generated by Algorithm~\ref{dbco:algorithm-one}. Then, for any comparator sequence $\bsy_T\in\calX_{T}$,
\begin{subequations}
\begin{align}
&\mathbf{E}[\Reg(\bsx_T,\bsy_T)]\nonumber\\
&\le \sum_{t=1}^T\mathbf{E}[d_2(t)]+C_{0}\sum_{t=1}^T\gamma_{t+1}
+\sum_{t=1}^T\sum_{i=1}^n\frac{p_i^2F_{f_i}^2\alpha_{i,t+1}}{\delta_{i,t}^2}\nonumber\\
&~~~+\sum_{i=1}^n\frac{2R_{i}^2}{\alpha_{i,T+1}}+\sum_{t=1}^{T-1}\sum_{i=1}^n
\frac{2R_{i}\|y_{i,t+1}-y_{i,t}\|}{\alpha_{i,t+1}}\nonumber\\
&~~~+\frac{1}{2}\sum_{t=1}^T\sum_{i=1}^n(4m\max_{i\in[n]}\{\frac{p_i^2F_{g_i}^2\alpha_{i,t+1}}{\delta_{i,t}^2}\}
+\frac{1}{\gamma_{t+1}}
-\frac{1}{\gamma_{t}}\nonumber\\
&~~~-\beta_{t+1})\mathbf{E}[\|q_{i,t}\|^2],\label{dbco:theoremregequ}\\
&\mathbf{E}[\|[\sum_{t=1}^Tg_{t}(x_{t})]_+\|^2]\nonumber\\
&\le d_5(T)\bigg\{\sum_{t=1}^T\mathbf{E}[d_2(t)]+C_{0}\sum_{t=1}^T\gamma_{t+1}\nonumber\\
&~~~+\sum_{t=1}^T\sum_{i=1}^n\frac{p_i^2F_{f_i}^2\alpha_{i,t+1}}{\delta_{i,t}^2}
+\sum_{i=1}^n\frac{2R_{i}^2}{\alpha_{i,T+1}}+2T\sum_{i=1}^{n}F_{f_i}\nonumber\\
&~~~+\frac{1}{2}\sum_{t=1}^T\sum_{i=1}^n(4m\max_{i\in[n]}\{\frac{p_i^2F_{g_i}^2\alpha_{i,t+1}}{\delta_{i,t}^2}\}
+\frac{1}{\gamma_{t+1}}
-\frac{1}{\gamma_{t}}\nonumber\\
&~~~-\beta_{t+1})\mathbf{E}[\|q_{i,t}-q^*\|^2]\bigg\},\label{dbco:theoremconsequ}
\end{align}
\end{subequations}
where $d_5(T)=2n(\frac{1}{\gamma_1}
+\sum_{t=1}^T(4m\max_{i\in[n]}\{\frac{p_i^2F_{g_i}^2\alpha_{i,t+1}}{\delta_{i,t}^2}\}+\beta_{t+1}))$ and  $q^*=\frac{2[\sum_{t=1}^Tg_{t}(x_{t})]_+}
{d_5(T)}\in\mathbb{R}^m_{+}$.
\end{lemma}
\begin{proof}
(a) For any $\lambda\in(0,1)$ and nonnegative sequence $\zeta_1,\zeta_2,\dots$, it holds that
\begin{align}\label{dbco:zeta}
\sum_{t=1}^T\sum_{s=1}^{t}\zeta_{s+1}\lambda^{t-s}
=\sum_{t=1}^{T}\zeta_{t+1}\sum_{s=0}^{T-t}\lambda^{s}
\le\frac{1}{1-\lambda}\sum_{t=1}^T\zeta_{t+1}.
\end{align}
Thus,
\begin{align}\label{dbco:zeta-d2}
\sum_{t=1}^Td_{1}(t)\le\frac{2\sqrt{m}n^2\tau B_1F_g}{1-\lambda}\sum_{t=1}^T\gamma_{t+1}.
\end{align}

The definition of $\Delta_t$ given by \eqref{dbco:delta} yields
\begin{align}
&-\sum_{t=1}^T\frac{\Delta_{t+1}}{2\gamma_{t+1}}\nonumber\\
&=\sum_{t=1}^T\frac{1}{2\gamma_{t+1}}\sum_{i=1}^n[(1-\beta_{t+1}\gamma_{t+1})\|q_{i,t}-q\|^2\nonumber\\
&~~~-\|q_{i,t+1}-q\|^2]\nonumber\\
&=\frac{1}{2}\sum_{i=1}^n\sum_{t=1}^T[\frac{1}{\gamma_{t}}\|q_{i,t}-q\|^2
-\frac{1}{\gamma_{t+1}}\|q_{i,t+1}-q\|^2]\nonumber\\
&~~~+\frac{1}{2}\sum_{t=1}^T\sum_{i=1}^n(\frac{1}{\gamma_{t+1}}
-\frac{1}{\gamma_{t}}-\beta_{t+1})\|q_{i,t}-q\|^2\nonumber\\
&=\frac{1}{2}\sum_{i=1}^n[\frac{1}{\gamma_{1}}\|q_{i,1}-q\|^2
-\frac{1}{\gamma_{T+1}}\|q_{i,T+1}-q\|^2]\nonumber\\
&~~~+\frac{1}{2}\sum_{t=1}^T\sum_{i=1}^n(\frac{1}{\gamma_{t+1}}
-\frac{1}{\gamma_{t}}-\beta_{t+1})\|q_{i,t}-q\|^2\nonumber\\
&\le\frac{n}{2\gamma_{1}}\|q\|^2+\frac{1}{2}\sum_{t=1}^T\sum_{i=1}^n(\frac{1}{\gamma_{t+1}}
-\frac{1}{\gamma_{t}}-\beta_{t+1})\|q_{i,t}-q\|^2,\label{dbco:qmu7}
\end{align}
where the last inequality holds since $q_{i,1}={\bf 0}_m$ and $\|q_{i,T+1}-q\|^2\ge0$.

From the properties of conditional expectation, we know that
\begin{align}\label{dbco:qmu8}
\mathbf{E}_{\calU_{T}}[\mathbf{E}_{\mathfrak{U}_{t}}[d_{4}(t)]]=\mathbf{E}[d_{4}(t)],~\forall t\in[T].
\end{align}

Noting that $\{\alpha_t\}$ is non-increasing and (\ref{dbco:domainupper}), for any $s\in[T]$, we have
\begin{align}
&\sum_{t=s}^Td_{4}(t)\nonumber\\
&=\frac{1}{2}\sum_{t=s}^T\sum_{i=1}^n(\frac{1}{\alpha_{i,t}}\|\check{y}_{i,t}-z_{i,t}\|^2
-\frac{1}{\alpha_{i,t+1}}\|\check{y}_{i,t+1}-z_{i,t+1}\|^2)\nonumber\\
&~~~+\frac{1}{2}\sum_{t=s}^T\sum_{i=1}^n(\frac{1}{\alpha_{i,t+1}}-\frac{1}{\alpha_{i,t}})
\|\check{y}_{i,t}-z_{i,t}\|^2\nonumber\\
&\le\frac{1}{2\alpha_{i,s}}\sum_{i=1}^n\|\check{y}_{i,s}-z_{i,s}\|^2\nonumber\\
&~~~-\frac{1}{2\alpha_{i,T+1}}\sum_{i=1}^n\|\check{y}_{i,T+1}-z_{i,T+1}\|^2\nonumber\\
&~~~+2\sum_{i=1}^n(\frac{1}{\alpha_{i,T+1}}
-\frac{1}{\alpha_{i,s}})R_{i}^2\le\sum_{i=1}^n\frac{2R_{i}^2}{\alpha_{i,T+1}}.\label{dbco:dyz}
\end{align}

Let $g_c:\mathbb{R}^m_{+}\rightarrow\mathbb{R}$ be a function defined as
\begin{align}\label{dbco:gc}
g_c(q)=(\sum_{t=1}^Tg_{t}(x_{t}))^\top q-\frac{d_5(T)}{4}\|q\|^2.
\end{align}

Combining (\ref{dbco:gvirtualnorm}) and (\ref{dbco:lemma_regretdeltaequ}), summing over $t\in[T]$, using (\ref{dbco:zeta-d2})--(\ref{dbco:gc})  and $g_{t}(y_t)\le{\bf 0}_{m},~\bsy_T\in\calX_{T}$, and taking expectation in $\calU_{T}$ yields
\begin{align}\label{dbco:theoremconsequ2}
&\mathbf{E}[g_c(q)]+\mathbf{E}[\Reg(\bsx_T,\bsy_T)]\nonumber\\
&\le \sum_{t=1}^T\mathbf{E}[d_2(t)]+C_{0}\sum_{t=1}^T\gamma_{t+1}
+\sum_{t=1}^T\sum_{i=1}^n\frac{p_i^2F_{f_i}^2\alpha_{i,t+1}}{\delta_{i,t}^2}\nonumber\\
&~~~+\sum_{i=1}^n\frac{2R_{i}^2}{\alpha_{i,T+1}}+\sum_{t=1}^T\sum_{i=1}^n
\frac{2R_{i}\|y_{i,t+1}-y_{i,t}\|}{\alpha_{i,t+1}}\nonumber\\
&~~~+\frac{1}{2}\sum_{t=1}^T\sum_{i=1}^n(4m\max_{i\in[n]}\{\frac{p_i^2F_{g_i}^2\alpha_{i,t+1}}{\delta_{i,t}^2}\}
+\frac{1}{\gamma_{t+1}}
-\frac{1}{\gamma_{t}}\nonumber\\
&~~~-\beta_{t+1})\mathbf{E}[\|q_{i,t}-q\|^2],~\forall q\in\mathbb{R}^m_{+}.
\end{align}
Then, substituting $q={\bf 0}_{m}$ into (\ref{dbco:theoremconsequ2}), setting $y_{i,T+1}=y_{i,T}$, and noting that $\{\alpha_t\}$ is non-increasing yields (\ref{dbco:theoremregequ}).

(b) Substituting $q=q^*$ into $g_c(q)$ gives
\begin{align}\label{dbco:gcequ}
g_c(q^*)=&\frac{\|[\sum_{t=1}^Tg_{t}(x_{t})]_+\|^2}
{d_5(T)}.
\end{align}
Moreover, (\ref{dbco:assfunction:ftgtupper}) gives
\begin{align}
|\Reg(\bsx_T,\bsy_T)|\le&2T\sum_{i=1}^{n}F_{f_i},~\forall \bsy_T\in\calX_T.\label{dbco:ff}
\end{align}
Substituting $q=q^*$ and $y_t=\check{x}^*_T,~t\in[T+1]$ into (\ref{dbco:theoremconsequ2}), combining (\ref{dbco:gcequ})--(\ref{dbco:ff}), and rearranging the terms gives (\ref{dbco:theoremconsequ}).
\end{proof}

We are now ready to prove Theorem~\ref{dbco:corollaryreg}.

(a) Applying (\ref{dbco:sequenceupp-1}), (\ref{dbco:sequenceupp}), and \eqref{dbco:lemma_virtualboundeqy} to the first three terms of the right-hand side of (\ref{dbco:theoremregequ}) and noting $\theta_2<\theta_3$ gives
\begin{subequations}
\begin{align}
&\sum_{t=1}^T\mathbf{E}[d_2(t)]\le\sum_{i=1}^n\frac{mF_gG_{g_i}(2r_i+R_{i})}{1-\theta_3+\theta_2}T^{1-\theta_3+\theta_2}\notag\\
&+\sum_{i=1}^n\frac{G_{f_i}(2r_i+R_{i})}{1-\theta_3}T^{1-\theta_3}
+\sum_{i=1}^n\frac{8mp_i^2F_{g_i}^2R^2_{i}}{r_i^2}\log(T),\label{dbco:corollaryregequ10}\\
&C_{0}\sum_{t=1}^T\gamma_{t+1}\le\frac{C_{0}}{\theta_2}T^{\theta_2},\label{dbco:corollaryregequ11}\\
&\sum_{t=1}^T\sum_{i=1}^n\frac{p_i^2F_{f_i}^2\alpha_{i,t+1}}{\delta_{i,t}^2}\notag\\
&\le\sum_{i=1}^n\frac{F_{f_i}^2}{4mF_{g_i}^2(1-\theta_1+2\theta_3)}T^{1-\theta_1+2\theta_3}.\label{dbco:corollaryregequ12}
\end{align}
\end{subequations}

From \eqref{dbco:stepsize1} and $\theta_1-2\theta_3\ge \theta_2$ we know that
\begin{align}\label{dbco:corollaryregequ12-2}
&4m\max_{i\in[n]}\{\frac{p_i^2F_{g_i}^2\alpha_{i,t+1}}{\delta_{i,t}^2}\}
+\frac{1}{\gamma_{t+1}}
-\frac{1}{\gamma_{t}}-\beta_{t+1}\notag\\
&=\frac{1}{(t+1)^{\theta_1-2\theta_3}}+\frac{t+1}{(t+1)^{\theta_2}}-\frac{t}{t^{\theta_2}}-
\frac{2}{(t+1)^{\theta_2}}\notag\\
&\le\frac{1}{(t+1)^{\theta_2}}+\frac{t+1}{(t+1)^{\theta_2}}-\frac{t}{t^{\theta_2}}-
\frac{2}{(t+1)^{\theta_2}}\notag\\
&=\frac{t}{(t+1)^{\theta_2}}-\frac{t}{t^{\theta_2}}<0.
\end{align}

Combining (\ref{dbco:theoremregequ}) and (\ref{dbco:corollaryregequ10})--(\ref{dbco:corollaryregequ12-2}) yields (\ref{dbco:corollaryregequ1}).

(b) Using (\ref{dbco:sequenceupp})  and noting $\theta_1-2\theta_3\ge \theta_2$ gives
\begin{align}
d_5(T)\le C_{2,1}T^{1-\theta_2}.\label{dbco:corollaryconsequ1}
\end{align}

Combining (\ref{dbco:theoremconsequ}) and (\ref{dbco:corollaryregequ10})--(\ref{dbco:corollaryconsequ1}) gives
\begin{align}
\mathbf{E}[\|[\sum_{t=1}^Tg_{t}(x_{t})]_+\|^2]\le C_{2}T^{2-\theta_2}.\label{dbco:corollaryconsequ2}
\end{align}

Finally, combining \eqref{dbco:corollaryconsequ2} and $(\mathbf{E}[\|[\sum_{t=1}^Tg_{t}(x_{t})]_+\|])^2\le \mathbf{E}[\|[\sum_{t=1}^Tg_{t}(x_{t})]_+\|^2]$ (which follows from Jensen's inequality) gives (\ref{dbco:corollaryconsequ}).

\subsection{Proof of Theorem~\ref{dbco:corollaryreg-two}}\label{dbco:corollaryreg-twoproof}
The proof is similar to the proof of Theorem~\ref{dbco:corollaryreg} with some modifications. Lemmas~\ref{dbco:lemma_virtualbound}--\ref{dbco:theoremreg} are replaced by the following Lemmas~\ref{dbco:lemma_virtualbound-two}--\ref{dbco:theoremreg-two}.

To simplify notation, we denote $\alpha_t=\alpha_{i,t}$, $\beta_t=\beta_{i,t}$, $\gamma_t=\gamma_{i,t}$, and $\xi_t=\xi_{i,t}$.

\begin{lemma}\label{dbco:lemma_virtualbound-two}
Suppose Assumptions \ref{dbco:assgraph}--\ref{dbco:assfunction:function} hold. For all $i\in[n]$ and $t\in\mathbb{N}_+$, $\tilde{q}_{i,t}$ and $q_{i,t}$ generated by Algorithm~\ref{dbco:algorithm-two} satisfy
\begin{subequations}
\begin{align}
&\|\tilde{q}_{i,t+1}\|\le \frac{B_1}{\beta_t},~\|q_{i,t}\|\le \frac{B_1}{\beta_t},\label{dbco:lemma_virtualboundeqy-two}\\
&\|\tilde{q}_{i,t+1}-\bar{q}_{t}\|
\le 2nB_1\tau \sum_{s=1}^{t-1}\gamma_{s+1}\lambda^{t-1-s},\label{dbco:lemma_qbarequ-two}\\
&\frac{\Delta_{t+1}}{2\gamma_{t+1}}\nonumber\\
&\le (\bar{q}_{t}-q)^\top g_{t}(x_{t})+2nB_1^2\gamma_{t+1}
+d_{6}(t)\nonumber\\
&~~~
+\frac{1}{2}\sum_{i=1}^n(2mp_i^2G_{g_i}^2\alpha_{t+1}+\beta_{t+1})\|q\|^2+d_{7}(t)
,\label{dbco:gvirtualnorm-two}
\end{align}
\end{subequations}
where
$q$ is an arbitrary vector in $\mathbb{R}^m_+$,
\begin{align*}
d_{6}(t)=2\sqrt{m}n^2B_1F_g\tau \sum_{s=1}^{t}\gamma_{s+1}\lambda^{t-s},
\end{align*}
and
\begin{align*}
d_{7}(t)=&\frac{1}{4\alpha_{t+1}}\sum_{i=1}^n\|x_{i,t+1}-x_{i,t}\|^2\\
&+\sum_{i=1}^n[\tilde{q}_{i,t+1}]^\top\hat{\nabla}_2g_{i,t}(x_{i,t})(x_{i,t+1}-x_{i,t}).
\end{align*}
\end{lemma}
\begin{proof}
From the fifth part in Lemma~\ref{dbco:lemma:uniformsmoothing} and \eqref{dbco:assfunction:functionLipg}, we know that for all $i\in[n]$, $x\in(1-\xi_{i,t})\mathbb{X}_i$, and $ t\in\mathbb{N}_+$,
\begin{align}\label{dbco:lemma_virtualboundgradg}
\|\hat{\nabla}_2 g_{i,t}(x)\|\le\sqrt{m}p_iG_{g_i}.
\end{align}
Hence, \eqref{dbco:algorithm-two:b}, \eqref{dbco:domainupper}, (\ref{dbco:assfunction:ftgtupper}), and \eqref{dbco:lemma_virtualboundgradg} yield
\begin{align}\label{dbco:lemma_qbarequb-two}
&\|c_{i,t+1}\|
\le\|g_{i,t}(x_{i,t})\|+\|\hat{\nabla}_2 g_{i,t}(x_{i,t})\|
\|(x_{i,t+1}-x_{i,t})\|\nonumber\\
&\le \sqrt{m}F_{g_i}+2\sqrt{m}p_iG_{g_i}R_{i}\le B_1,~\forall i\in[n],~\forall t\in\mathbb{N}_+.
\end{align}

Replacing $z_{i,t}$ and $g_{i,t}(z_{i,t})$ by $x_{i,t}$ and $c_{i,t+1}$, respectively, and following steps similar to those used to prove (\ref{dbco:lemma_virtualboundeqy}) and \eqref{dbco:lemma_qbarequ} yields (\ref{dbco:lemma_virtualboundeqy-two}) and \eqref{dbco:lemma_qbarequ-two}.

Applying (\ref{dbco:lemma:projection:none}) to (\ref{dbco:algorithm-two:q}) yields
\begin{align}
&\|q_{i,t}-q\|^2
\le\Big\|(1-\beta_t\gamma_t)\tilde{q}_{i,t}+\gamma_tc_{i,t}-q\Big\|^2\nonumber\\
&=\|\tilde{q}_{i,t}-q\|^2+\gamma_t^2\|c_{i,t}-\beta_t\tilde{q}_{i,t}\|^2\nonumber\\
&~~~+2\gamma_t[\tilde{q}_{i,t}]^\top\hat{\nabla}_2 g_{i,t-1}(x_{i,t-1})(x_{i,t}-x_{i,t-1})\nonumber\\
&~~~-2\gamma_tq^\top\hat{\nabla}_2g_{i,t-1}(x_{i,t-1})(x_{i,t}-x_{i,t-1})\nonumber\\
&~~~+2\gamma_t[\tilde{q}_{i,t}-q]^\top g_{i,t-1}(x_{i,t-1})\nonumber\\
&~~~-2\beta_t\gamma_t[\tilde{q}_{i,t}-q]^\top\tilde{q}_{i,t},\label{dbco:qmu-two}
\end{align}
For the fourth term of the right-hand side of (\ref{dbco:qmu-two}),  \eqref{dbco:lemma_virtualboundgradg} and the Cauchy-Schwarz  inequality yield
\begin{align}
&-2\gamma_tq^\top\hat{\nabla}_2g_{i,t-1}(x_{i,t-1})(x_{i,t}-x_{i,t-1})\nonumber\\
&\le2\gamma_t(mp^2_iG^2_{g_i}\alpha_t\|q\|^2
+\frac{1}{4\alpha_t}\|x_{i,t}-x_{i,t-1}\|^2).\label{dbco:qmu2-two}
\end{align}

Replacing \eqref{dbco:qmu} by \eqref{dbco:qmu-two}, using \eqref{dbco:qmu2-two},
and following steps similar to those used to prove \eqref{dbco:gvirtualnorm} yields \eqref{dbco:gvirtualnorm-two}.
\end{proof}

\begin{lemma}\label{dbco:lemma_regretdelta-two}
Suppose Assumptions \ref{dbco:assgraph}--\ref{dbco:assfunction:function} hold. For all $i\in[n]$, let $\{x_{t}\}$ be the sequence generated by Algorithm~\ref{dbco:algorithm-two} and $\{y_t\}$ be an arbitrary sequence in $\mathbb{X}$, then
\begin{align}
&f_{t}(x_{t})-f_{t}(y_{t})\nonumber\\
&\le (\bar{q}_{t})^\top (g_{t}(y_{t})- g_{t}(x_{t}))+2d_{6}(t)-E_{\mathfrak{U}_{t}}[d_{7}(t)]\nonumber\\
&~~~+\sum_{i=1}^np_i^2G^2_{f_i}\alpha_{t+1}
+\sum_{i=1}^n\frac{2R_{i}\|y_{i,t+1}-y_{i,t}\|}{\alpha_{t+1}}\nonumber\\
&~~~+d_8(t)+\mathbf{E}_{\mathfrak{U}_{t}}[d_{9}(t)],~\forall t\in\mathbb{N}_+,\label{dbco:lemma_regretdeltaequ-two}
\end{align}
where
\begin{align*}
d_8(t)=&\sum_{i=1}^n\Big\{(\delta_{i,t}+R_{i}\xi_{t})
(\sqrt{m}G_{g_i}\|q_{i,t}\|+G_{f_i})\\
&~~~~~~~+\frac{2R^2_{i}(\xi_{t}-\xi_{t+1})}{\alpha_{t+1}}\Big\},\\
d_{9}(t)=&\frac{1}{2\alpha_{t+1}}\sum_{i=1}^n(\|\check{y}_{i,t}-x_{i,t}\|^2
-\|\check{y}_{i,t+1}-x_{i,t+1}\|^2),
\end{align*}
and $\check{y}_{i,t}=(1-\xi_{t})y_{i,t}$.
\end{lemma}
\begin{proof}
Replacing $z_{i,t}$, $a_{i,t}$, and \eqref{dbco:lemma_regretdeltaequ:fsmooth3} by $x_{i,t}$, $b_{i,t}$, and
\begin{align}
\|\hat{\nabla}_2f_{i,t}(x)\|\le p_iG_{f_i},
\label{dbco:lemma_regretdeltaequ:fsmooth3-two}
\end{align}
respectively,
deleting \eqref{dbco:qmu2}, and following steps similar to those used to prove \eqref{dbco:lemma_regretdeltaequ} yields \eqref{dbco:lemma_regretdeltaequ-two}.
\end{proof}

\begin{lemma}\label{dbco:theoremreg-two}
Suppose Assumptions~\ref{dbco:assgraph}--\ref{dbco:assfunction:function} hold. For any $T\in\mathbb{N}_+$, let $\bsx_T$ be the sequence generated by Algorithm~\ref{dbco:algorithm-two}. Then, for any comparator sequence $\bsy_T\in\calX_{T}$,
\begin{subequations}
\begin{align}
&\mathbf{E}[\Reg(\bsx_T,\bsy_T)]\nonumber\\
&\le \sum_{t=1}^T\mathbf{E}[d_8(t)]+ \hat{C}_{0}\sum_{t=1}^T\gamma_{t+1}+\sum_{i=1}^n\frac{2R_{i}^2}{\alpha_{T+1}}
\nonumber\\
&~~~+\frac{1}{2}\sum_{t=1}^T\sum_{i=1}^n(\frac{1}{\gamma_{t+1}}
-\frac{1}{\gamma_{t}}-\beta_{t+1})\mathbf{E}[\|q_{i,t}\|^2]\nonumber\\
&~~~+\sum_{t=1}^T\sum_{i=1}^np_i^2G^2_{f_i}\alpha_{t+1}+\frac{2R_{\max}V(\bsy_T)}{\alpha_{T}},\label{dbco:theoremregequ-two}\\
&\mathbf{E}[\|[\sum_{t=1}^Tg_{t}(x_{t})]_+\|^2]\nonumber\\
&\le d_{10}(T)\bigg\{\sum_{t=1}^T\mathbf{E}[d_8(t)]+\hat{C}_{0}\sum_{t=1}^T\gamma_{t+1}+\sum_{i=1}^n\frac{2R_{i}^2}{\alpha_{T+1}}\nonumber\\
&~~~+\frac{1}{2}\sum_{t=1}^T\sum_{i=1}^n(\frac{1}{\gamma_{t+1}}
-\frac{1}{\gamma_{t}}-\beta_{t+1})\mathbf{E}[\|q_{i,t}-\hat{q}^*\|^2]\nonumber\\
&~~~+\sum_{t=1}^T\sum_{i=1}^np_i^2G^2_{f_i}\alpha_{t+1}+2T\sum_{i=1}^nF_{f_i}\bigg\},\label{dbco:theoremconsequ-two}
\end{align}
\end{subequations}
where $d_{10}(T)=2n(\frac{1}{\gamma_1}
+\sum_{t=1}^T(2mp_i^2G_{g_i}^2\alpha_{t+1}+\beta_{t+1}))$
and $\hat{q}^*=\frac{2[\sum_{t=1}^Tg_{t}(x_{t})]_+}
{d_{10}(T)}\in\mathbb{R}^m_{+}$.
\end{lemma}
\begin{proof}
With Lemmas~\ref{dbco:lemma_virtualbound-two} and \ref{dbco:lemma_regretdelta-two} at hand, the proof of Lemma~\ref{dbco:theoremreg-two} follows steps similar to those used to prove Lemma~\ref{dbco:theoremreg}.
\end{proof}

Finally, with Lemmas~\ref{dbco:lemma_virtualbound-two}--\ref{dbco:theoremreg-two} at hand, the proof of \eqref{dbco:corollaryregequ1-two} and \eqref{dbco:corollaryconsequ-two} follows steps similar to those used to prove \eqref{dbco:corollaryregequ1} and \eqref{dbco:corollaryconsequ}.
\end{document}